\documentclass[mathpazo]{cicp}

\begin{document}
\title{Kernel-free Boundary Integral Method for Allen--Cahn and Cahn--Hilliard Equations on Irregular Domains}



\author[Liu X et al.]{Xinru Liu\affil{1},
        Yulin Zhang\affil{1},
        Wenjun Ying\affil{2}, and
         Pensong Yin\affil{1}\comma\corrauth}

\address{\affilnum{1}\ School of Mathematical Sciences and Institute of Natural Sciences,
          Shanghai Jiao Tong University,
          Minhang, Shanghai 200240, P. R. China. \\\\
          \affilnum{2}\ School of Mathematical Sciences, MOE-LSC and Institute of Natural Sciences,
          Shanghai Jiao Tong University,
          Minhang, Shanghai 200240, P. R. China.}

\emails{{\tt yps2701155@sjtu.edu.cn} (P. Yin)}

\begin{abstract}
A unified kernel-free boundary integral (KFBI) framework is proposed for the Allen--Cahn and Cahn--Hilliard equations with homogeneous no-flux boundary conditions on two- and three-dimensional irregular domains. With a stabilized first-order implicit--explicit (IMEX) discretization, the Allen--Cahn update reduces to a Neumann modified Helmholtz problem. An auxiliary-variable reformulation reduces the Cahn--Hilliard update to subproblems of the same type without evaluating the Laplacian of the nonlinear source. The Cahn--Hilliard update uses two sequential solves for real shifts, while complex-conjugate shifts allow the real-valued solution to be reconstructed from a single complex solve. All subproblems are handled by the same KFBI solver that indirectly evaluates potentials through equivalent interface problems. 

For the Cahn--Hilliard equation, a geometry-weighted discrete $L^2$ projection is incorporated into the KFBI update to enforce the prescribed mass constraint. The resulting correction is spatially uniform and is computed using only values at physical grid nodes.

Numerical experiments with manufactured solutions demonstrate second-order spatial convergence and first-order convergence of the paired temporal-error indicators. Source-free simulations show discrete energy decay for both equations. In the Cahn--Hilliard computations, the mass projection maintains the prescribed discrete mass and causes only a small overall energy change compared with the observed dissipation.

\end{abstract}


\maketitle

\section{Introduction}
\label{sec1}
Phase-field models describe interfacial evolution through diffuse transition layers and are widely used in phase separation, solidification, grain growth, microstructure evolution, and multiphase flows \cite{allen1979microscopic,cahn1958free,chen2002phase,jacqmin1999calculation}. The Allen--Cahn and Cahn--Hilliard equations are the nonconservative \(L^2\) and conservative \(H^{-1}\) gradient flows of the Ginzburg-Landau free energy, respectively. Under homogeneous no-flux boundary conditions, both equations dissipate free energy, while the source-free Cahn--Hilliard equation also conserves mass.

The small interfacial parameter and nonlinear potential make phase-field equations stiff, motivating convex-splitting \cite{eyre1998unconditionally} and stabilized semi-implicit or implicit-explicit (IMEX) schemes \cite{he2007large,shen2010numerical,tang2016implicit}. In this work, we adopt a stabilized first-order IMEX discretization, which reduces the Allen--Cahn and Cahn--Hilliard updates to modified Helmholtz and modified bi-Helmholtz problems, respectively. The main computational task is therefore to solve these elliptic subproblems efficiently on irregular domains.

Common spatial discretizations of the phase-field equations include finite-difference, Fourier spectral, finite-element, finite-volume, and boundary integral methods. On rectangular or periodic domains, finite-difference and spectral methods provide efficient high-order approximations and fast solvers \cite{furihata2001stable,zhu1999coarsening}. For solid-state dewetting, Jiang et al. applied a stabilized semi-implicit cosine pseudospectral method to a variable-mobility Cahn--Hilliard model \cite{jiang2012phase}. On complex domains, finite-element and finite-volume methods can use boundary-fitted meshes, while mixed finite-element formulations reduce the fourth-order Cahn--Hilliard equation to coupled second-order problems \cite{elliott1987numerical,barrett1999finite}. Discontinuous Galerkin methods avoid globally $C^1$-continuous trial spaces through numerical fluxes \cite{wells2006discontinuous}. Boundary integral, embedded-boundary, immersed, cut-cell, and diffuse-domain methods provide alternatives that avoid body-fitted volume meshes \cite{johansen1998embedded,li2009diffuse}. Wei et al.~\cite{wei2020integral} addressed the two-dimensional Cahn--Hilliard wetting problem using an integral equation method, with a volume-potential reformulation that avoids numerical differentiation of the nonlinear source. 

For the Cahn--Hilliard equation, exact mass conservation at the fully discrete level requires compatibility among the spatial operator, mass quadrature, and no-flux boundary treatment. It can be achieved through flux cancellation in finite-difference and finite-volume methods, constant-function testing in finite elements, conservative numerical fluxes in discontinuous Galerkin methods, or preservation of the zero Fourier mode in spectral methods \cite{furihata2001stable,elliott1987numerical,wells2006discontinuous,cances2021finitevolume}. On irregular domains, Shin et al. used a boundary-control function to construct a conservative finite-difference method \cite{shin2011conservative}, while Kim et al. combined a related geometric treatment with linear IMEX Runge-Kutta schemes \cite{kim2023linear}.  These methods enforce conservation through the spatial discretization without a separate mass correction.

The kernel-free boundary integral (KFBI) method evaluates volume and boundary potentials through equivalent interface problems on Cartesian grids, avoiding explicit Green's functions and direct singular quadrature while retaining fast elliptic solvers \cite{ying2007kernel,ying2013kernel,xie2019biharmonic,zhou2024correction}. In the phase-field setting, Zeng and Xie \cite{zeng2025allen} applied the KFBI method to the Allen--Cahn equation on two-dimensional irregular domains with homogeneous Dirichlet boundary conditions, using volume and double-layer potentials. 

The present work develops a unified KFBI framework for the Allen--Cahn and Cahn--Hilliard equations with homogeneous no-flux conditions on two- and three-dimensional irregular domains. Within this framework, both equations are reduced to modified Helmholtz problems with homogeneous Neumann boundary conditions. Moreover, an auxiliary-variable reformulation of the Cahn--Hilliard update avoids explicitly evaluating the Laplacian of the nonlinear source. The resulting modified Helmholtz problems have either positive real or complex-conjugate shifts, both handled by a common KFBI procedure. In the complex-conjugate case, the real-valued phase field is reconstructed from the result of one KFBI solve with a complex shift.

The time-discrete, space-continuous Cahn--Hilliard scheme conserves mass. After spatial discretization, however, interface corrections and boundary-trace reconstruction in the KFBI implementation need not preserve the geometry-weighted discrete mass, and mass defects may accumulate over time. Qiao et al.~\cite{qiao2025massprojection} enforced mass conservation for stabilized radial basis function finite-difference schemes on closed surfaces using a surface-area-weighted $L^2$ projection. Motivated by this approach, we construct a geometry-weighted mass projection tailored to the KFBI method on irregular domains. Cut-control-volume weights associated with exterior nodes are redistributed to neighboring physical nodes, preserving the total discrete area or volume while excluding auxiliary extension values from the mass calculation. Using the same weights in the mass constraint and the discrete $L^2$ norm yields a spatially uniform correction of minimum weighted norm that restores the prescribed discrete mass without solving an additional elliptic problem.

The two main contributions are:
\begin{enumerate}
    \item A unified KFBI framework is developed for no-flux Allen--Cahn and Cahn--Hilliard equations on two- and three-dimensional irregular domains. For the Cahn--Hilliard update, an auxiliary-variable reduction avoids evaluating the Laplacian of the nonlinear source. Subproblems are handled by the same KFBI solver used for the Allen--Cahn update. The formulation requires two sequential solves for real shifts, while complex-conjugate shifts allow the real-valued update to be reconstructed from a single complex KFBI computation. 
    \item For the Cahn--Hilliard equation, a geometry-weighted discrete $L^2$ projection is incorporated into the KFBI time-stepping scheme to correct mass defects arising from spatial discretization. Applied after each provisional KFBI update, the projection restores the prescribed discrete mass through a spatially uniform correction without solving an additional elliptic problem. 
\end{enumerate}

Numerical experiments in two and three dimensions assess spatial accuracy, convergence of paired temporal-error indicators, mass conservation, and energy behavior, including the perturbation introduced by the projection.

The remainder of this paper is organized as follows. Section~2 introduces the model problems. Section~3 presents the stabilized first-order IMEX discretization. Section~4  derives the Neumann reductions and their common boundary integral formulation, and Section~5 describes the KFBI method. Section~6 introduces geometry-weighted mass projection for the Cahn--Hilliard equation, and Section~7 summarizes the complete algorithms. Section~8 presents numerical experiments and Section~9 gives the conclusion.

\section{Model Problems}
Let $\Omega\subset \mathbb{R}^d\,(d=2,3)$ be a  bounded domain with a smooth and irregular boundary $\partial \Omega$. Consider the Allen--Cahn equation with a no-flux boundary condition
\begin{subequations}\label{Allen--Cahn}
\begin{align}
    &\partial_t u=\varepsilon\Delta u-\varepsilon^{-1}f(u),&\text{in}\,\,&\Omega_T:=\Omega\times (0,T),\\
    &u(\boldsymbol{x},0)=u^0(\boldsymbol{x}),&\text{in}\,\,&\Omega,\\
    &\partial_{\boldsymbol{n}}u=0,&\text{on}\,\,&\partial\Omega_T:=\partial\Omega\times(0,T),
\end{align}
\end{subequations}
where $\Delta$ denotes the Laplace operator, $\boldsymbol{n}$ is the outward unit normal to $\partial\Omega$, $F(u)=\frac{1}{4}(u^2-1)^2$ denotes the double-well potential, $f(u)=F'(u)$, and $\partial_{\boldsymbol{n}}$ denotes the outward normal derivative on $\partial\Omega$.

Similarly, we seek $u=u(\boldsymbol{x},t)$ satisfying the Cahn--Hilliard equation with homogeneous no-flux boundary conditions
\begin{subequations}\label{cahn-hilliard}
\begin{align}
       & \partial_t u=\Delta\mu,&\text{in}\,\,&\Omega_T,\\
    &\mu=\varepsilon^{-1}f(u)-\varepsilon\Delta u,&\text{in}\,\,&\Omega_T,\\
    &u(\boldsymbol{x},0)=u^0(\boldsymbol{x}),&\text{in}\,\,&\Omega,\\
    &\partial_{\boldsymbol{n}}u=\partial_{\boldsymbol{n}}\mu=0,&\text{on}\,\,&\partial\Omega_T.
\end{align}
\end{subequations} 

\section{Stabilized First-Order IMEX Scheme}

Let $t_n=n\tau$, where $\tau>0$ is the time step, and let $s>0$ be a dimensionless stabilization parameter with $S=s/\varepsilon$. Following Rothe's method, first-order stabilized implicit-explicit (IMEX) schemes for the Allen--Cahn and Cahn--Hilliard equations were proposed in \cite{gao2012gradient, shen2010numerical} \begin{subequations}\label{SSI-Allen--Cahn}
    \begin{align}
        &\frac{u^{n+1}-u^n}{\tau}-\varepsilon\Delta u^{n+1}+\varepsilon^{-1}f(u^n)+S(u^{n+1}-u^n)=0,\,\,&\text{in}\,\,&\Omega,\\
        &\partial_{\boldsymbol{n}}u^{n+1}=0,\,\,&\text{on}\,\,&\partial\Omega,
    \end{align}
\end{subequations} and \begin{subequations}\label{SSI-cahn-hilliard}
    \begin{align}
            &\frac{u^{n+1}-u^n}{\tau}=\Delta \mu^{n+1},\,\,&\text{in}\,\,&\Omega,\\
        &\mu^{n+1}=-\varepsilon\Delta u^{n+1}+\varepsilon^{-1}f(u^n)+S(u^{n+1}-u^n),\,\,&\text{in}\,\,&\Omega,
    \end{align}
\end{subequations} with boundary conditions \begin{subequations}\label{ssi-bound-cond}
    \begin{align}
        \partial_{\boldsymbol{n}}u^{n+1}&=0,\,\,&\text{on}\,\,&\partial\Omega,\\
        \partial_{\boldsymbol{n}}\mu^{n+1}&=0,\,\,&\text{on}\,\,&\partial\Omega.
    \end{align}
\end{subequations}

Following \cite{tang2020efficient}, consider the energy decomposition
 \begin{equation*}
    \mathcal{E}(u)=\int_{\Omega}\left(\frac{\varepsilon}{2}|\nabla u|^2+\frac{S}{2}u^2\right)\,\mathrm{d}\boldsymbol{x}-\int_{\Omega}\left(\frac{S}{2}u^2-\varepsilon^{-1}F(u)\right)\,\mathrm{d}\boldsymbol{x}:=\mathcal{E}^+(u)-\mathcal{E}^-(u).    
\end{equation*}
The functional $\mathcal{E}^{+}$ is convex and $\mathcal{E}^{-}$ is convex on the set of functions satisfying $|u|\leq\sqrt{(s+1)/3}$. Therefore, the above decomposition provides a conditional convex-splitting interpretation under a suitable bound on the numerical solution. In the remainder of this paper, we refer to \eqref{SSI-Allen--Cahn} and \eqref{SSI-cahn-hilliard} as stabilized IMEX schemes.

For the Cahn--Hilliard equation, integrating the first equation in \eqref{SSI-cahn-hilliard} over $\Omega$ and using the no-flux boundary condition gives
\begin{equation*}
\int_{\Omega}u^{n+1}\,\mathrm{d}\boldsymbol{x}=\int_{\Omega}u^n\,\mathrm{d}\boldsymbol{x}=\cdots=\int_\Omega u^0\,\mathrm{d}\boldsymbol{x},\,\,\forall n\geq 0,
\end{equation*}
whereas the Allen--Cahn scheme is nonconservative. 

Furthermore, Tang and Yang \cite{tang2016implicit} proved that, if the IMEX scheme (\ref{SSI-Allen--Cahn}) is applied for time discretization and second-order centered differences are used for spatial discretization, then the resulting numerical approximation of the Allen--Cahn equation (\ref{Allen--Cahn}) satisfies the discrete maximum bound principle, provided that $s\geq2$ and $\|u^0\|_{L^\infty}\leq1$. Specifically,
\begin{equation*}
    \|u^{n+1}\|_{L^\infty}\leq1,
    \qquad \forall\,\tau>0,\quad \forall\,n\geq0.
\end{equation*}
Under these conditions, the finite-difference scheme is also energy stable.

For the Cahn--Hilliard equation \eqref{cahn-hilliard}, He, Liu, and Tang \cite{he2007large} showed that the time-discrete, space-continuous stabilized IMEX scheme \eqref{SSI-cahn-hilliard} satisfies the energy-dissipation inequality $\mathcal{E}(u^{n+1})\leq \mathcal{E}(u^n)$ for all $n\geq 0$, if 
\begin{equation*}
\begin{aligned}
s \geq {}&
\max_{\boldsymbol{x}\in\Omega}
\left\{
    \frac12|u^n(\boldsymbol{x})|^2
    +\frac14
    |u^{n+1}(\boldsymbol{x})+u^n(\boldsymbol{x})|^2
\right\}
-\frac12,
\qquad \forall\,n\geq0.
\end{aligned}
\end{equation*}
This sufficient condition depends on the numerical solution at two consecutive time levels. In this work, we set $s=2$ in all reported experiments. The energy behavior of the fully discrete KFBI-IMEX scheme, including the mass-projected Cahn--Hilliard implementation, is assessed numerically in Section~8.

To facilitate the application of the KFBI method, we rewrite equations (\ref{SSI-Allen--Cahn}) and (\ref{SSI-cahn-hilliard}), (\ref{ssi-bound-cond}) in the following forms \begin{subequations}\label{IMEX-Allen--Cahn}
    \begin{align}
        &\left(\frac{1}{\tau}+\varepsilon^{-1}s-\varepsilon\Delta\right)u^{n+1}=\left(\frac{1}{\tau}+\varepsilon^{-1}s\right)u^n-\varepsilon^{-1}f(u^n),\,\,&\text{in}\,\,&\Omega,\\
       & \partial_{\boldsymbol{n}}u^{n+1}=0,\,\,&\text{on}\,\,&\partial\Omega,
    \end{align}
\end{subequations} and \begin{subequations}\label{IMEX-cahn-hilliard}
    \begin{align}
        (\Delta^2-b\Delta+c)u^{n+1}&=f_1(\boldsymbol{x}),&\text{in}\,\,&\Omega,\label{IMEX-cahn-hilliard-a}\\
        (\Delta-b)u^{n+1}+\varepsilon^{-1}\mu^{n+1}&=f_2(\boldsymbol{x}),&\text{in}\,\,&\Omega,\label{IMEX-cahn-hilliard-b}\\
        \partial_{\boldsymbol{n}}u^{n+1}&=0,&\text{on}\,\,&\partial\Omega,\label{IMEX-cahn-hilliard-c}\\
        \partial_{\boldsymbol{n}}\mu^{n+1}&=0,&\text{on}\,\,&\partial\Omega,\label{IMEX-cahn-hilliard-d}
    \end{align}
\end{subequations} where $b=s\varepsilon^{-2},c=\tau^{-1}\varepsilon^{-1}$, and \begin{subequations}
    \begin{align}
        f_1(\boldsymbol{x})&=cu^n(\boldsymbol{x})+\Delta f_2(\boldsymbol{x}),\\
        f_2(\boldsymbol{x})&=\frac{(u^n(\boldsymbol{x}))^3-(1+s)u^n(\boldsymbol{x})}{\varepsilon^2}.
    \end{align}
\end{subequations}

\section{Neumann Reduction and Boundary Integral Formulation}

This section reformulates the time-discrete Allen--Cahn and Cahn--Hilliard equations as modified Helmholtz problems with homogeneous Neumann boundary conditions and derives corresponding boundary integral equations.

\subsection{Green's Function}
Let $\Omega\subset\mathbb{R}^d$ be a bounded irregular domain with boundary $\partial\Omega$, and let $\mathcal{B}$ be a box embedding $\Omega$. Let $G(\boldsymbol{x},\boldsymbol{y})$ be Green's function associated with the modified Helmholtz operator $\mathcal{L}=\Delta-\lambda^2$ on $\mathcal{B}$ such that for any fixed $\boldsymbol{y}\in\mathcal{B}$, 
    \begin{align*}
        \mathcal{L}G(\boldsymbol{x},\boldsymbol{y})&=\delta(\boldsymbol{x}-\boldsymbol{y}),\,\,&\text{in}\,\,&\mathcal{B},\\
        G(\boldsymbol{x},\boldsymbol{y})&=0,&\text{on}\,\,&\partial\mathcal{B},
    \end{align*}
where $\delta(\cdot)$ is the Dirac delta function.

\subsection{Boundary Integral Equation for the Allen--Cahn Update}

For the Allen--Cahn update, \eqref{IMEX-Allen--Cahn} gives 
    \begin{align*}
        (\Delta-\lambda^2)u^{n+1}&=q^n,\,\,&\text{in}\,\,&\Omega,\\
        \partial_{\boldsymbol{n}}u^{n+1}&=0,&\text{on}\,\,&\partial\Omega,
    \end{align*}
where $\lambda^2=\varepsilon^{-2}s+\tau^{-1}\varepsilon^{-1}$ and $ q^n=\varepsilon^{-2}f(u^n)-\lambda^2u^n$. Thus each Allen--Cahn step requires one Neumann modified Helmholtz solve with a positive real shift. The solution can be represented by 
\begin{equation*}\label{u-ac}
    u^{n+1}(\boldsymbol{x})=\int_{\Omega}G(\boldsymbol{x},\boldsymbol{y})q^n(\boldsymbol{y})\,\mathrm{d}\boldsymbol{y}-\int_{\partial\Omega}G(\boldsymbol{x},\boldsymbol{y})\psi(\boldsymbol{y})\,\mathrm{d}s_{\boldsymbol{y}},
\end{equation*} where the density $\psi$ satisfies the following boundary integral equation \begin{equation*}\label{BIE-ac}
    \frac{1}{2}\psi(\boldsymbol{x})-\frac{\partial}{\partial\boldsymbol{n}}\int_{\partial\Omega}G(\boldsymbol{x},\boldsymbol{y})\psi(\boldsymbol{y})\,\mathrm{d}s_{\boldsymbol{y}}=-\frac{\partial}{\partial\boldsymbol{n}}\int_\Omega G(\boldsymbol{x},\boldsymbol{y})q^n(\boldsymbol{y})\,\mathrm{d}\boldsymbol{y},
\end{equation*} which is a Fredholm integral equation of the second kind.

\subsection{Boundary Integral Equation for the Cahn--Hilliard Update}

\noindent\textbf{Auxiliary-variable reduction.}
For the Cahn--Hilliard update, define 
\begin{equation*}\label{eq:ch-shifts}
\begin{gathered}
b=\frac{s}{\varepsilon^2},\qquad
c=\frac{1}{\varepsilon\tau},\qquad D=b^2-4c,\\
\lambda_1^2=\begin{cases}
(b-\sqrt D)/2,&D\geq0,\\
(b+\mathrm{i}\sqrt{-D})/2,&D<0,
\end{cases}
\qquad\lambda_2^2=b-\lambda_1^2,\qquad \alpha=\lambda_1^2,\qquad\beta=\lambda_2^2.
\end{gathered}
\end{equation*}
$\lambda_i^2$, $i=1,2$, are the roots of $x^2-bx+c=0$. At the current time step, set
\begin{equation*}\label{eq:ch-nonlinear-source}
f_2=\frac{f(u^n)-s u^n}{\varepsilon^2}
=\frac{(u^n)^3-(1+s)u^n}{\varepsilon^2}.
\end{equation*}
Then equation~\eqref{IMEX-cahn-hilliard-a} becomes
\begin{equation*}\label{eq:ch-factorized-update}
(\Delta-\lambda_1^2)(\Delta-\lambda_2^2) u^{n+1}
=f_1,\qquad f_1=cu^n+\Delta f_2.
\end{equation*}
Introduce $z=(\Delta-\lambda_2^2)u^{n+1}-f_2
=\lambda_1^2u^{n+1}-\varepsilon^{-1}\mu^{n+1}$.  The second equality follows from \eqref{IMEX-cahn-hilliard-b} and $\lambda_1^2+\lambda_2^2=b$. Applying $\Delta-\lambda_1^2$ to $z$ gives $(\Delta-\lambda_1^2)z=cu^n+\lambda_1^2 f_2$. Hence the Laplacian of the nonlinear source cancels before spatial discretization.  Moreover, the no-flux conditions \eqref{ssi-bound-cond} at the new time step imply $\partial_{\boldsymbol n}z
=\lambda_1^2\partial_{\boldsymbol n}u^{n+1}
-\varepsilon^{-1}\partial_{\boldsymbol n}\mu^{n+1}
=0.$ The fourth-order update therefore reduces to
\begin{align}
&(\Delta-\lambda_1^2)z=cu^n+\lambda_1^2 f_2,
&\partial_{\boldsymbol n}z=0,
\label{ch-first-shift}\\
&(\Delta-\lambda_2^2) u^{n+1}=z+f_2,
&\partial_{\boldsymbol n} u^{n+1}=0.
\label{ch-second-shift}
\end{align}

\noindent\textbf{First shifted problem.}
Using Green's function $G_1$ associated with $\Delta-\lambda_1^2$, represent the solution as
\begin{equation}\label{eq:ch-first-representation}
z(\boldsymbol x)
=\int_\Omega G_1(\boldsymbol x,\boldsymbol y)
\bigl(cu^n(\boldsymbol y)+\lambda_1^2 f_2(\boldsymbol y)\bigr)\,\mathrm d\boldsymbol y
-\int_{\partial\Omega} G_1(\boldsymbol x,\boldsymbol y)
\psi_1(\boldsymbol y)\,\mathrm d s_{\boldsymbol y},
\qquad\boldsymbol x\in\Omega.
\end{equation}
The density $\psi_1$ satisfies the following boundary integral equation
\begin{equation}\label{eq:ch-first-bie}
\frac12\psi_1(\boldsymbol x)
-\frac{\partial}{\partial\boldsymbol n}\int_{\partial\Omega}
G_1(\boldsymbol x,\boldsymbol y)
\psi_1(\boldsymbol y)\,\mathrm d s_{\boldsymbol y}
=-\frac{\partial}{\partial\boldsymbol n}
\int_\Omega G_1(\boldsymbol x,\boldsymbol y)
\bigl(cu^n(\boldsymbol y)+\lambda_1^2 f_2(\boldsymbol y)\bigr)\,\mathrm d\boldsymbol y,
\qquad\boldsymbol x\in\partial\Omega.
\end{equation}

\noindent\textbf{Second shifted problem.}
With $z$ determined, use Green's function $G_2$ associated with $\Delta-\lambda_2^2$ to write
\begin{equation}\label{eq:ch-second-representation}
 u^{n+1}(\boldsymbol x)
=\int_\Omega G_2(\boldsymbol x,\boldsymbol y)
\bigl(z(\boldsymbol y)+f_2(\boldsymbol y)\bigr)
\,\mathrm d\boldsymbol y
-\int_{\partial\Omega} G_2(\boldsymbol x,\boldsymbol y)
\psi_2(\boldsymbol y)\,\mathrm d s_{\boldsymbol y},
\qquad\boldsymbol x\in\Omega.
\end{equation}
The density $\psi_2$ satisfies the following boundary integral equation
\begin{equation}\label{eq:ch-second-bie}
\frac12\psi_2(\boldsymbol x)
-\frac{\partial}{\partial\boldsymbol n}\int_{\partial\Omega}
G_2(\boldsymbol x,\boldsymbol y)
\psi_2(\boldsymbol y)\,\mathrm d s_{\boldsymbol y}
=-\frac{\partial}{\partial\boldsymbol n}
\int_\Omega G_2(\boldsymbol x,\boldsymbol y)
\bigl(z(\boldsymbol y)+f_2(\boldsymbol y)\bigr)
\,\mathrm d\boldsymbol y,
\quad\boldsymbol x\in\partial\Omega.
\end{equation}
Each density is determined by a Fredholm integral equation of the second kind.

Substituting \eqref{eq:ch-first-representation} into \eqref{eq:ch-second-representation} yields
\begin{equation}\label{eq:ch-expanded-integrals}
\begin{aligned}
 u^{n+1}(\boldsymbol x)
&=\int_\Omega G_2(\boldsymbol x,\boldsymbol y)
\left[\int_\Omega G_1(\boldsymbol y,\boldsymbol\xi)
\bigl(cu^n(\boldsymbol\xi)+\lambda_1^2 f_2(\boldsymbol\xi)\bigr)
\,\mathrm d\boldsymbol\xi\right]\mathrm d\boldsymbol y+\int_\Omega G_2(\boldsymbol x,\boldsymbol y)
f_2(\boldsymbol y)\,\mathrm d\boldsymbol y\\
&-\int_\Omega G_2(\boldsymbol x,\boldsymbol y)
\int_{\partial\Omega}G_1(\boldsymbol y,\boldsymbol\xi)
\psi_1(\boldsymbol\xi)\,\mathrm d s_{\boldsymbol\xi}
\,\mathrm d\boldsymbol y
-\int_{\partial\Omega}G_2(\boldsymbol x,\boldsymbol y)
\psi_2(\boldsymbol y)\,\mathrm d s_{\boldsymbol y},
\end{aligned}
\end{equation}
for $\boldsymbol x\in\Omega$. For simplicity,  define
\begin{equation*}\label{eq:common-potentials}
\begin{aligned}
\mathcal{Y}_i f(\boldsymbol x)
=\int_\Omega G_i(\boldsymbol x,\boldsymbol y)
f(\boldsymbol y)\,\mathrm d\boldsymbol y,\qquad
\mathcal{S}_i\psi(\boldsymbol x)
=\int_{\partial\Omega}G_i(\boldsymbol x,\boldsymbol y)
\psi(\boldsymbol y)\,\mathrm d s_{\boldsymbol y},
\qquad i=1,2.
\end{aligned}
\end{equation*}
Here $\mathcal{Y}_i$ and $\mathcal{S}_i$ denote the Yukawa volume potential and single-layer potential operators, respectively. Equation~\eqref{eq:ch-expanded-integrals} can then be rewritten as
\begin{equation*}\label{eq:ch-expanded-recovery}
 u^{n+1}
=\mathcal{Y}_2\bigl(\mathcal{Y}_1(cu^n+\lambda_1^2f_2)\bigr)
-\mathcal{Y}_2(\mathcal{S}_1\psi_1)
+\mathcal{Y}_2f_2-\mathcal{S}_2\psi_2.
\end{equation*}
The density equations can be written entirely in terms of the known source and the densities
\begin{subequations}\label{BIE-ch}
\begin{align}
\frac12\psi_1-\partial_{\boldsymbol n}(\mathcal{S}_1\psi_1)
&=-\partial_{\boldsymbol n}\bigl[\mathcal{Y}_1(cu^n+\lambda_1^2f_2)\bigr],
\label{eq:ch-first-bie-compact}\\
\frac12\psi_2-\partial_{\boldsymbol n}(\mathcal{S}_2\psi_2)
&=-\partial_{\boldsymbol n}\Bigl[
\mathcal{Y}_2\bigl(\mathcal{Y}_1(cu^n+\lambda_1^2f_2)\bigr)
-\mathcal{Y}_2(\mathcal{S}_1\psi_1)
+\mathcal{Y}_2 f_2\Bigr].
\label{eq:ch-second-bie-compact}
\end{align}
\end{subequations}

\subsection{Real and Complex-Conjugate Shifts}

The type of shifts is determined by the sign of the discriminant:  $D=b^2-4c
=\frac{s^2}{\varepsilon^4}
\left(1-\frac{\tau_c}{\tau}\right),$ $\tau_c=\frac{4\varepsilon^3}{s^2}.$ Only the shifts $\alpha=\lambda_1^2$ and $\beta=\lambda_2^2$ are required by the KFBI solver.

\noindent\textbf{Positive real shifts ($D\geq0$):
} When $\tau>\tau_c$, the two shifts are real, positive, and distinct. First solve \eqref{eq:ch-first-bie} and recover $z$ from \eqref{eq:ch-first-representation}. Then solve \eqref{eq:ch-second-bie} and recover $u^{n+1}$ from \eqref{eq:ch-second-representation}. When $\tau=\tau_c$ ($D=0$), the two shifts coincide. The same shifted Neumann operator is applied successively to two different right-hand sides. 

\noindent\textbf{Complex-conjugate shifts ($D<0$)}: When $0<\tau<\tau_c$, denote $\alpha=a+\mathrm{i}\eta$, $\beta=a-\mathrm{i}\eta,$ and $\eta=\sqrt{-D}/2>0$. Since $u^{n+1}$ and $f_2$ are real, the auxiliary-variable definition gives $z=(\Delta-\beta)u^{n+1}-f_2=(\Delta-a)u^{n+1}-f_2+\mathrm{i}\eta u^{n+1}$. Taking imaginary parts yields $\operatorname{Im}z=\eta u^{n+1}$, and hence
\begin{equation}\label{eq:ch-complex-reconstruction}
u^{n+1}
=\frac{\operatorname{Im}z}{\eta}
=\frac{\operatorname{Im}z}{\operatorname{Im}\alpha}.
\end{equation}
Solve only \eqref{eq:ch-first-bie} in complex arithmetic and recover $z$ through \eqref{eq:ch-first-representation}. The real phase field is then recovered directly from the imaginary part of $z$, eliminating the second complex computation.

\section{The Kernel-Free Boundary Integral Method}

Section 4 reduces the Allen--Cahn and Cahn--Hilliard updates to Neumann modified Helmholtz problems and derives their boundary integral equations. Following the KFBI approach, the single-layer boundary and volume integrals are evaluated indirectly through equivalent interface problems.


\subsection{Equivalent Interface Problems}
The Yukawa and single-layer potentials can be evaluated through the interface problems described below. The jump across the interface $\partial\Omega$ is denoted by 
\begin{align*}
	\begin{split}
[\![ u ]\!](\boldsymbol{x})=\lim_{\boldsymbol{y}\in\Omega,\boldsymbol{y}\to\boldsymbol{x}}u(\boldsymbol{y})-\lim_{\boldsymbol{y}\in\Omega^c,\boldsymbol{y}\to\boldsymbol{x}}u(\boldsymbol{y}),\,\boldsymbol{x}\in\partial\Omega.
	\end{split}
\end{align*}
 As stated in \cite{ying2007kernel, ying2013kernel}, the Yukawa and boundary potentials can be expressed as equivalent interface problems in a general form, \begin{align}\label{interface-problem}
\begin{split}
(\Delta-\lambda^2)u&= F,\,\,\;\;\text{in}\,\,\mathcal{B}\setminus\partial\Omega,\\
[\![ u ]\!]&=\Phi, \,\,\;\;\text{on}\,\,\partial\Omega,\\
[\![ \partial_{\boldsymbol{n}}u]\!]&=\Psi,\,\,\;\;\text{on}\,\,\partial\Omega,\\
u&=0, \,\,\;\,\,\,\text{on}\,\,\partial\mathcal{B}.
\end{split}	
\end{align} For the Yukawa potential $\mathcal{Y}f$, $F$ equals $f$ in $\Omega$ and vanishes outside $\Omega$, while $\Phi=\Psi=0$. For the single-layer potential $\mathcal{S}\psi$, $F=\Phi=0,\Psi=-\psi$. The same formulation applies to a complex shift $\lambda^2$, with the Cartesian-grid solve and the reconstructed traces carried out in complex arithmetic.

To evaluate the Yukawa and single-layer potentials in \eqref{BIE-ac} and \eqref{BIE-ch}, the KFBI method replaces the direct kernel-based evaluation by solving the equivalent interface problem \eqref{interface-problem}. This avoids the explicit use of Green's functions and direct singular quadrature.

 \subsection{Cartesian Grid Method for the Interface Problem}

The embedding box is discretized by a uniform Cartesian grid. A second-order central finite difference scheme is used at grid nodes. At irregular nodes, where the finite difference stencil crosses the interface, correction terms are constructed from the prescribed jumps in the solution and its normal derivative. Since the interface corrections modify only the right-hand side, the coefficient matrix remains identical to that of the standard modified Helmholtz problem on the rectangular embedding box. The resulting Cartesian-grid system can therefore be solved efficiently using an FFT-based elliptic solver. 

After the interface problem is solved, the boundary values and normal derivatives required by the boundary integral formulation are reconstructed locally from nearby grid values using polynomial interpolation incorporating the interface jump data. Differentiating this polynomial in the outward normal direction gives the interior normal trace needed by the Neumann formulation.

The resulting boundary integral equation is discretized at boundary nodes and solved by the generalized minimal residual (GMRES) method, with each evaluation of the discrete boundary integral operator performed by the Cartesian-grid solver. Consequently, the Allen--Cahn solver, the real-shift Cahn--Hilliard solver, and the complex-shift Cahn--Hilliard solver share the same spatial-solver architecture.

\section{Geometry-Weighted Mass Projection for the KFBI Method}

This section introduces a geometry-weighted mass projection for the Cahn--Hilliard KFBI-IMEX scheme. Cut-control-volume weights are transferred from exterior to physical nodes, so that the discrete mass uses only physical-node values. The weighted $L^2$ projection then restores the prescribed mass through a spatially uniform correction, without solving an additional elliptic problem.

\subsection{Discrete Mass Defect of the Provisional KFBI Update}

Let $\mathcal J_h$ denote the index set of all grid nodes, and let $I_h=\{i\in\mathcal J_h:\boldsymbol x_i\in\Omega\}$ be the physical-node index set. The provisional KFBI-IMEX step is 
\begin{equation*}\label{eq:mass-provisional-update}
 \widetilde U_h^{\,n+1}=S_{h,\tau}(U_h^n),
 \qquad
 \widetilde u_h^{\,n+1}
 =\left.\widetilde U_h^{\,n+1}\right|_{I_h}.
\end{equation*}
Here $S_{h,\tau}$ is the provisional KFBI-IMEX solution operator. Uppercase symbols denote fields on the embedding grid, whereas lowercase symbols denote their restrictions to $I_h$.

Let $\omega_i>0$, $i\in I_h$, be the geometric weights, and define
\begin{equation*}\label{eq:mass-functional-kfbi}
 M_h(v_h)=\sum_{i\in I_h}\omega_i v_i,
 \qquad
 M_h^\star=M_h(u_h^0).
\end{equation*}
The target $M_h^\star$ is fixed for the source-free problem and the provisional mass defect is $d_h^{n+1}
 =M_h(\widetilde u_h^{\,n+1})-M_h^\star.$ After spatial discretization, the KFBI interface corrections and boundary-trace reconstruction do not automatically satisfy a discrete Green's identity with respect to $M_h$. Thus the provisional update need not preserve the geometry-weighted mass, and its defects may accumulate over successive time steps. We impose the mass constraint after the provisional elliptic solve:
\begin{equation*}\label{eq:combined-overview}
 U_h^n
 \xrightarrow{\;S_{h,\tau}\;}
 \widetilde U_h^{\,n+1},
 \qquad
 u_h^{n+1}
 =P_h\left(\left.\widetilde U_h^{\,n+1}\right|_{I_h}\right).
\end{equation*}

\subsection{Cut-Control-Volume Weights and Redistribution}

Let $\Omega=\{\boldsymbol x:\phi(\boldsymbol x)<0\}$ be the fixed physical domain. On a Cartesian grid of uniform spacing $h$, the nodal control volume is $ V_i=\mathcal B\cap
 \prod_{\ell=1}^{d}
 \left[x_{i,\ell}-\frac h2,x_{i,\ell}+\frac h2\right],$ $d=2,3.$ The preliminary geometric weight is the area or volume of the portion of $V_i$ inside the reconstructed domain $ \widehat\omega_i=|V_i\cap\Omega_h|$, $i\in\mathcal J_h$, where $\Omega_h$ is obtained by piecewise-linear interpolation of the level-set function, and $|\cdot|$ denotes area in two dimensions and volume in three dimensions.

To compute $\widehat\omega_i$, each control volume is divided into four triangles in two dimensions or twelve tetrahedra in three dimensions. Within each simplex, the portion where the interpolated level-set function is non-positive is retained. Summing the retained areas or volumes gives the geometric weight.

A control volume may intersect $\Omega_h$ even when its center lies outside $\Omega$. Its positive overlap measure must be included in the geometric quadrature, but its nodal value is an auxiliary KFBI extension value. To obtain a mass functional supported only on physical nodes, define the set of exterior contributors $ E_h=\{i\in\mathcal J_h\setminus I_h:
          \widehat\omega_i>0\}.$ For each $i\in E_h$, set $  r_i=\min\left\{r\in\{1,2,3\}:\mathcal N_i(r)\neq\varnothing\right\}$ and $\mathcal N_i(r)=\left\{j\in I_h:\|j-i\|_\infty=r\right\}$. The overlap measure is distributed equally among the nodes in the nearest ring $\mathcal N_i(r_i)$: 
\begin{equation*}\label{eq:mass-redistribution-coefficients}
 \theta_{ij}
 =\begin{cases}
   |\mathcal N_i(r_i)|^{-1},&j\in\mathcal N_i(r_i),\\
   0,&\text{otherwise},
  \end{cases}
 \qquad
 \sum_{j\in I_h}\theta_{ij}=1.
\end{equation*}
The final physical-node weights are
\begin{equation}\label{eq:mass-final-weights}
 \omega_j
 =\widehat\omega_j
  +\sum_{i\in E_h}\theta_{ij}\widehat\omega_i,
 \qquad j\in I_h.
\end{equation}
Furthermore,
\begin{equation*}\label{eq:mass-measure-preservation}
 \begin{aligned}
 A_h=\sum_{j\in I_h}\omega_j
=\sum_{j\in I_h}\widehat\omega_j
   +\sum_{i\in E_h}\widehat\omega_i
       \sum_{j\in I_h}\theta_{ij}
 =\sum_{i\in\mathcal J_h}\widehat\omega_i
 =|\Omega_h|.
 \end{aligned}
\end{equation*}
Thus the discrete mass retains all cut-control-volume contributions while using only physical-node values. The same weights define the weighted-$L^2$ projection in the next subsection. For the fixed domain, the weights and $A_h$ are computed
once and reused at every time step.

\subsection{Geometry-Weighted $L^2$ Mass Projection}

The mass correction is chosen to minimize the change in the physical
phase field in the same geometric metric used to measure its mass.
Using the weights in \eqref{eq:mass-functional-kfbi}, introduce the geometry-weighted inner product and norm
\begin{equation*}\label{eq:mass-weighted-metric}
 (v_h,w_h)_{h,\omega}
 =\sum_{i\in I_h}\omega_i v_iw_i,
 \qquad
 \|v_h\|_{h,\omega}^2=(v_h,v_h)_{h,\omega},
\end{equation*}
and the affine constraint set
\begin{equation*}\label{eq:mass-affine-set}
 \mathcal A_h=\{v_h:M_h(v_h)=M_h^\star\}.
\end{equation*}
The corrected solution is defined as the weighted-$L^2$ projection of the provisional field onto $\mathcal A_h$
\begin{equation*}\label{eq:mass-projection-minimization}
 u_h^{n+1}=P_h\widetilde u_h^{\,n+1}
 =\operatorname*{arg\,min}_{v_h\in\mathcal A_h}
       \frac12\|v_h-\widetilde u_h^{\,n+1}\|_{h,\omega}^2.
\end{equation*}
Because every physical weight is positive, this minimizer is unique. Introducing a Lagrange multiplier $\lambda$, the stationarity conditions give
\begin{equation*}
  \omega_i\bigl(v_i-\widetilde u_{h,i}^{\,n+1}\bigr)
  -\lambda\omega_i=0,\qquad i\in I_h.
\end{equation*}
Hence the minimum-norm correction is spatially constant, and the mass constraint determines it as
\begin{equation*}
  \lambda=\delta_h^{n+1}=
  \frac{M_h^\star-M_h(\widetilde u_h^{\,n+1})}{A_h}.
\end{equation*}
Thus $\delta_h^{n+1}$ is the constant correction (equivalently, the Lagrange multiplier) generated by the weighted-$L^2$ projection. The projection applies this constant to every physical degree of freedom
\begin{equation*}
  u_{h,i}^{n+1}=(P_h\widetilde u_h^{\,n+1})_i
  =\widetilde u_{h,i}^{\,n+1}+\delta_h^{n+1},\qquad i\in I_h.
\end{equation*}

In exact arithmetic, mass conservation follows by direct expansion:
\begin{equation*}
  \begin{aligned}
  M_h(u_h^{n+1})
  &=\sum_{i\in I_h}\omega_i
    \left(\widetilde u_{h,i}^{\,n+1}+\delta_h^{n+1}\right)\\
  &=M_h(\widetilde u_h^{\,n+1})
    +\delta_h^{n+1}M_h(1)\\
  &=M_h(\widetilde u_h^{\,n+1})
    +\delta_h^{n+1}A_h=M_h^\star.
  \end{aligned}
\end{equation*}
In floating-point arithmetic, the remaining mass error is affected by summation and rounding errors. The implementation applies the same constant to the full embedding-grid field $ U_{h,i}^{n+1}
 =\widetilde U_{h,i}^{\,n+1}+\delta_h^{n+1}
  $, $i\in\mathcal J_h.$ The extension values do not enter the mass calculation; they are updated only to provide a synchronized full-grid field for the subsequent KFBI-IMEX step. The correction requires only a weighted summation and
a uniform addition, with no additional elliptic solve.

The following proposition shows that the mass correction is stable
in the geometry-weighted norm and does not increase the error
relative to a mass-compatible discrete reference solution.

\begin{proposition}[Stability of the mass projection]
Assume that $\omega_i>0$ for all $i\in I_h$. For any discrete
field $v_h$ and any $w_h\in\mathcal A_h$, the mass projection
satisfies
\begin{equation}
\|P_hv_h-w_h\|_{h,\omega}^2
=
\|v_h-w_h\|_{h,\omega}^2
-\frac{|M_h(v_h)-M_h^\star|^2}{A_h}.
\label{eq:projection-stability}
\end{equation}
Consequently,
\begin{equation}
\|P_hv_h-w_h\|_{h,\omega}
\leq
\|v_h-w_h\|_{h,\omega},
\label{eq:projection-nonexpansive}
\end{equation}
and
\begin{equation*}
\|P_hv_h-v_h\|_{h,\omega}
=
\frac{|M_h(v_h)-M_h^\star|}{\sqrt{A_h}}.
\label{eq:projection-correction-norm}
\end{equation*}
\end{proposition}

\begin{proof}
Denote $\mathbf 1_h=(1,1,\cdots,1)^T$. Since $w_h\in\mathcal A_h$, we have
$M_h(w_h)=M_h^\star$. Let $e_h=v_h-w_h$. Using
\[
P_hv_h
=
v_h-\frac{M_h(v_h)-M_h^\star}{A_h}\mathbf 1_h,
\qquad
A_h=\|\mathbf 1_h\|_{h,\omega}^2,
\]
we obtain
\[
P_hv_h-w_h
=
e_h-\frac{(e_h,\mathbf 1_h)_{h,\omega}}{A_h}\mathbf 1_h.
\]
Therefore,
\[
\begin{aligned}
\|P_hv_h-w_h\|_{h,\omega}^2
&=
\|e_h\|_{h,\omega}^2
-\frac{|(e_h,\mathbf 1_h)_{h,\omega}|^2}{A_h}\\
&=
\|v_h-w_h\|_{h,\omega}^2
-\frac{|M_h(v_h)-M_h^\star|^2}{A_h}.
\end{aligned}
\]
This proves \eqref{eq:projection-stability} and
\eqref{eq:projection-nonexpansive}. Moreover,
\[
P_hv_h-v_h
=
-\frac{M_h(v_h)-M_h^\star}{A_h}\mathbf 1_h,
\]
which gives \eqref{eq:projection-correction-norm}.
\end{proof} As a direct consequence, if a mass-compatible discrete reference
solution $w_h^{n+1}\in\mathcal A_h$ satisfies
\[
\|\widetilde u_h^{\,n+1}-w_h^{n+1}\|_{h,\omega}
\leq C(h^p+\tau^q),
\]
then
\[
\|P_h\widetilde u_h^{\,n+1}-w_h^{n+1}\|_{h,\omega}
\leq C(h^p+\tau^q).
\]
Hence the mass projection does not reduce the spatial or temporal
order of the provisional approximation relative to a
mass-compatible discrete reference solution.

The provisional stabilized IMEX update is energy dissipative under the corresponding stabilization condition at the time-discrete, space-continuous level. The subsequent constant mass projection leaves the gradient-energy contribution unchanged but introduces a sign-indefinite perturbation to the bulk energy. Section~8 quantifies the energy perturbation together with the mass defects and the observed energy decay.

\section{Algorithm Summary}

Algorithm~\ref{alg:Framework} combines the common Neumann KFBI solver,
the Allen--Cahn and Cahn--Hilliard updates, and the geometry-weighted
mass projection into a complete time-stepping procedure for the
source-free problems. 
Let $W_h=\mathcal K_h(\kappa,r_h)$ denote a Neumann KFBI solver for
$(\Delta-\kappa)w=r$ in $\Omega$ with
$\partial_{\boldsymbol n}w=0$ on $\partial\Omega$, where $r_h$ is
the source sampled at the physical nodes. The output $W_h$ is the
field on the embedding grid, and $W_h|_{I_h}$ approximates the
physical solution. $\mathbf 1_{\mathcal J_h}$
denotes the constant-one field on the embedding grid.

\begin{algorithm}[!ht]
  \caption{Unified KFBI-IMEX scheme for the source-free
  Allen--Cahn and Cahn--Hilliard equations.}
  \label{alg:Framework}
  \small
  \begin{algorithmic}
    \Require Initial field
    $U_h^0$, parameters $\varepsilon,s,\tau$, and number of steps $N_t$.
    \State Construct the Cartesian grid and boundary discretization;
    identify physical and irregular nodes.
    \State Compute the shift parameters for the selected model.
    \State $u_h^0\gets U_h^0|_{I_h}$.
    \If{Cahn--Hilliard model}
      \State Compute the geometric weights
      $\{\omega_i\}_{i\in I_h}$ .
      \State $A_h\gets\sum_{i\in I_h}\omega_i$,
      \quad $M_h^\star\gets M_h(u_h^0)$.
    \EndIf
    \For{$n=0,\ldots,N_t-1$}
      \If{Allen--Cahn model}
        \State $q_h^n\gets\varepsilon^{-2}f(u_h^n)-\lambda^2u_h^n$.
        \State $U_h^{n+1}\gets\mathcal K_h(\lambda^2,q_h^n)$.
      \Else
        \State $f_{2,h}^n\gets
        \varepsilon^{-2}\bigl(f(u_h^n)-s u_h^n\bigr)$.
        \State $Z_h\gets
        \mathcal K_h(\alpha,c u_h^n+\alpha f_{2,h}^n)$.
        \If{$D\geq0$}
          \State $\widetilde U_h^{\,n+1}\gets
          \mathcal K_h(\beta,Z_h|_{I_h}+f_{2,h}^n)$.
        \Else
          \State $\widetilde U_h^{\,n+1}\gets
          \operatorname{Im}Z_h/\operatorname{Im}\alpha$.
        \EndIf
        \State $\widetilde u_h^{\,n+1}\gets
        \widetilde U_h^{\,n+1}|_{I_h}$.
        \State $\delta_h^{n+1}\gets
        \bigl(M_h^\star-M_h(\widetilde u_h^{\,n+1})\bigr)/A_h$.
        \State $U_h^{n+1}\gets\widetilde U_h^{\,n+1}
        +\delta_h^{n+1}\mathbf 1_{\mathcal J_h}$.
      \EndIf
      \State $u_h^{n+1}\gets U_h^{n+1}|_{I_h}$.
    \EndFor
  \end{algorithmic}
\end{algorithm}


The manufactured sources and projection targets used in the numerical tests are specified in Section~8. Prescribed source terms are incorporated through the right-hand sides of the shifted elliptic subproblems, without changing the KFBI solution procedure. For the Cahn--Hilliard equation, the mass-projection target is chosen consistently with the prescribed mass evolution.

\section{Numerical Results}
\label{sec;numerical}

The proposed method is tested on two- and three-dimensional irregular domains. The Allen--Cahn solver is assessed in Examples~1-4, and the mass-projected Cahn--Hilliard solver in Examples~5-8. Spatial errors and paired temporal-error indicators are evaluated using manufactured solutions. Phase evolution, energy dissipation, and Cahn--Hilliard mass conservation are examined in source-free simulations.

The embedding box is fixed as $\mathcal{B}=[-1,1]^d$, $d=2,3$. A rotated ellipse and an ellipsoid are used for the manufactured-solution tests; a rotated five-fold star-shaped domain and a torus are used for the source-free simulations. The geometries are defined in the corresponding examples. In all tests, the stabilization parameter and GMRES tolerance are set to $s=2$ and $10^{-10}$, respectively. The average total number of GMRES iterations per time step is denoted by $\overline{k}_{\rm G}$. The corresponding averages for the smaller and larger shifts are denoted by $\overline{k}_{\rm G}^{\rm s}$ and $\overline{k}_{\rm G}^{\rm l}$. In the paired temporal tests, the counts are reported as $\overline{k}_{\rm G}^{\rm cont}/\overline{k}_{\rm G}^{\rm disc}$ for the continuous-source and discrete-compatible computations, respectively.

Let $I_h$ be the index set of Cartesian grid nodes in the physical domain,
where $h$ is the mesh width. For a grid function
$v_h=\{v_{h,i}\}_{i\in I_h}$, the discrete error norms are defined by 
\[
\|v_h\|_{2,h}
=\left(\sum_{i\in I_h}\omega_i|v_{h,i}|^2\right)^{1/2},
\qquad
\|v_h\|_{\infty,h}
=\max_{i\in I_h}|v_{h,i}|.
\]
Spatial accuracy and temporal consistency are assessed separately using  a continuous source and a discrete-compatible source from a prescribed exact solution $u_{\rm ex}$, with $U^n(\boldsymbol{x})=u_{\rm ex}(\boldsymbol{x},t_n)$. The continuous source is obtained by substituting $u_{\rm ex}$ into the governing equation 
\[
g_{\rm cont}^n=
\begin{cases}
\partial_t u_{\rm ex}(t_n)-\varepsilon\Delta U^n
+\varepsilon^{-1}f(U^n),
&\text{Allen--Cahn},\\
\partial_t u_{\rm ex}(t_n)+\varepsilon\Delta^2 U^n
-\varepsilon^{-1}\Delta f(U^n),
&\text{Cahn--Hilliard}.
\end{cases}
\]
The discrete-compatible source is obtained by substituting $U^n$ and $U^{n+1}$ into the stabilized IMEX scheme
\[
g_{\rm disc}^n=
\begin{cases}
\displaystyle
\frac{U^{n+1}-U^n}{\tau}-\varepsilon\Delta U^{n+1}
+\varepsilon^{-1}f(U^n)
+\frac{s}{\varepsilon}(U^{n+1}-U^n),
&\text{Allen--Cahn},\\
\displaystyle
\frac{U^{n+1}-U^n}{\tau}+\varepsilon\Delta^2 U^{n+1}
-\varepsilon^{-1}\Delta f(U^n)
-\frac{s}{\varepsilon}\bigl(\Delta U^{n+1}-\Delta U^n\bigr),
&\text{Cahn--Hilliard}.
\end{cases}
\]
Here, the spatial derivatives are evaluated analytically. With $g_{\rm disc}^n$, the time-discrete, space-continuous scheme is satisfied exactly by $U^n$, so the temporal consistency defect is removed from the spatial tests. 

For the spatial convergence tests, the solution computed with $g_{\rm disc}^n$ is compared with $u_{\rm ex}(T)$ at the final time $T$. The errors and observed orders are defined by 
\[
E_q(h)=\|u_h(T)-u_{\rm ex}(T)\|_{q,h},
\qquad
\operatorname{ord}_q(h)
=\frac{\log E_q(2h)-\log E_q(h)}{\log 2},
\qquad q\in\{2,\infty\}.
\]

For the paired temporal tests,  two solutions, $u_{h,\tau}^{\rm cont}$ and $u_{h,\tau}^{\rm disc}$, are computed on a fixed grid for each $\tau$ using $g_{\rm cont}^n$ and $g_{\rm disc}^n$, respectively. The paired temporal-error indicator and its observed order are defined by 
\[
E_{q,{\rm pair}}(\tau)
=\|u_{h,\tau}^{\rm cont}(T)-u_{h,\tau}^{\rm disc}(T)\|_{q,h},
\qquad \operatorname{ord}_{q,{\rm pair}}(\tau)
=\frac{\log E_{q,{\rm pair}}(2\tau)-\log
E_{q,{\rm pair}}(\tau)}{\log 2},\qquad q\in\{2,\infty\}.
\]
The time-discrete, space-continuous solution driven by $g_{\rm cont}^n$ is denoted by $v_\tau(T)$. The two numerical solutions can be written as $u_{h,\tau}^{\rm cont}(T)=v_\tau(T)+e_h^{\rm cont}$ and $u_{h,\tau}^{\rm disc}(T)=u_{\rm ex}(T)+e_h^{\rm disc},$ where $e_h^{\rm cont}$ and $e_h^{\rm disc}$ are the spatial errors. Their difference is therefore $u_{h,\tau}^{\rm cont}(T)-u_{h,\tau}^{\rm disc}(T)
=
\bigl(v_\tau(T)-u_{\rm ex}(T)\bigr)
+
\bigl(e_h^{\rm cont}-e_h^{\rm disc}\bigr).$ 
Pairing is intended to reduce spatial-error contributions common to the two computations. The remaining difference, \(e_h^{\rm cont}-e_h^{\rm disc}\), need not vanish. When this difference is small relative to the temporal error over the tested time-step range, first-order convergence of the paired indicators provides numerical evidence consistent with the first-order IMEX discretization.

For the source-free simulations, the discrete energy is defined by
\[
E_h(u_h)=\sum_{i\in I_h}\omega_i
\left[
\frac{\varepsilon}{2}\left|\nabla_h^\Omega u_i\right|^2
+\frac{1}{4\varepsilon}(u_i^2-1)^2
\right].
\]
The discrete gradient \(\nabla_h^\Omega u_i\) is evaluated using physical grid nodes only. Centered differences are used in the interior, and inward one-sided differences are used near the irregular boundary to avoid differentiating the auxiliary extension outside the physical domain. 

The physical-domain mass and measure are defined by $M_h(v_h)=\sum_{i\in I_h}\omega_i v_{h,i}$ and $A_h=M_h(1)$. In the Cahn--Hilliard manufactured-solution tests, the projection target is set to $M_h^\star(t_{n+1})=M_h(U^{n+1})$ so that the sampled exact solution satisfies the discrete mass constraint. In the source-free tests, the initial mass $M_h^\star=M_h(u_h^0)$ is prescribed. For nonzero $M_h^\star$, the relative mass error is measured by
\[
e_M(v_h)=\frac{|M_h(v_h)-M_h^\star|}{|M_h^\star|}.
\]
The provisional and projected states are denoted by $\widetilde u_h^{\,n+1}$ and $u_h^{n+1}$. The energy change caused by each projection and its magnitude relative to the initial energy are defined by
\[
\Delta E_{\mathrm{proj}}^{n+1}=E_h(u_h^{n+1})-E_h(\widetilde u_h^{\,n+1}),\qquad
r_{\mathrm{proj}}^{n+1}=\frac{|\Delta E_{\mathrm{proj}}^{n+1}|}{E_h(u_h^0)}\times100\%.
\]
When the net energy decrease is positive, the cumulative perturbation is measured relative to this decrease
\[
R_{\mathrm{proj}}=\frac{\sum_{n=0}^{N_t-1}|\Delta E_{\mathrm{proj}}^{n+1}|}{E_h(u_h^0)-E_h(u_h^{N_t})}\times100\%,\qquad N_t=T/\tau.
\]
These measures are used in Examples~6 and~8. The gradient energy is unchanged by the constant correction, but the bulk energy may be increased or decreased. Energy dissipation is therefore checked numerically; it is not guaranteed by mass conservation alone.

\subsection{Geometric Quadrature Accuracy}

The geometric quadrature is checked on the rotated ellipse (\ref{eq:rotated-ellipse}) in Example 1 and ellipsoid (\ref{eq:ac-ellipsoid-3d}) in Example~3 by integrating the constant function $v_1=1$ and the polynomials $v_2=1+x^2+2y^2+0.1x$ and $v_3=1+x^2+2y^2+3z^2+0.1x$, respectively. Define the relative quadrature error and quadrature order by 
\[e_{\rm quad}(v;h)=
\frac{\left|M_h(v_h)-\int_\Omega v\,\mathrm{d}\boldsymbol{x}\right|}
{\left|\int_\Omega v\,\mathrm{d}\boldsymbol{x}\right|},\qquad
\operatorname{ord}_{\rm quad}(v;h)
=
\frac{\log\!\left(
e_{\rm quad}(v;2h)/e_{\rm quad}(v;h)
\right)}{\log 2}.
\] Second-order convergence is observed for both the domain measure and polynomial integrals in Table~\ref{tab:geometric-quadrature-accuracy}. All final physical-node weights are positive, and weights are zero on the exterior grid nodes.

\begin{table}[!htbp]
\centering
\small
\renewcommand{\arraystretch}{1.1}
\begin{tabular*}{\linewidth}
{@{\extracolsep{\fill}}crrrr@{}}
\hline
$N$ & $e_{\rm quad}(1;h)$ & $\operatorname{ord}_{\rm quad}(1;h)$
    & $e_{\rm quad}(v_d;h)$ & $\operatorname{ord}_{\rm quad}(v_d;h)$\\
\hline
\multicolumn{5}{c}{Two dimensions: rotated ellipse.}\\
128  & 1.11e-04 & --   & 1.99e-04 & --\\
256  & 2.86e-05 & 1.96 & 4.91e-05 & 2.02\\
512  & 7.11e-06 & 2.01 & 1.27e-05 & 1.95\\
1024 & 1.77e-06 & 2.01 & 3.26e-06 & 1.96\\
2048 & 4.45e-07 & 1.99 & 8.09e-07 & 2.01\\
\hline
\multicolumn{5}{c}{Three dimensions: ellipsoid.}\\
64   & 1.67e-03 & --   & 2.39e-03 & --\\
128  & 4.16e-04 & 2.00 & 5.95e-04 & 2.01\\
256  & 1.04e-04 & 2.00 & 1.51e-04 & 1.98\\
512  & 2.60e-05 & 2.00 & 3.78e-05 & 2.00\\
\hline
\end{tabular*}
\caption{Relative errors and observed orders of geometric quadrature.}
\label{tab:geometric-quadrature-accuracy}
\end{table}

\subsection{Examples for the Allen--Cahn Equation}
The following four examples examine the spatial accuracy, convergence of paired temporal-error indicators, phase evolution, and energy dissipation of the Allen--Cahn solver.


{\bf Example 1.}
This example considers the forced two-dimensional Allen--Cahn equation with a manufactured source. The rotated ellipse
\begin{equation}\label{eq:rotated-ellipse}
\begin{aligned}
\Omega &= \left\{(x,y):\frac{\xi^2}{a^2}+\frac{\eta^2}{b^2}<1\right\},\\
\xi &= x\cos\alpha+y\sin\alpha,\qquad
\eta=-x\sin\alpha+y\cos\alpha,\\
a&=0.8,\qquad b=0.5,\qquad \alpha=\pi/12,
\end{aligned}
\end{equation}
is embedded in $\mathcal{B}=[-1,1]^2$. The exact solution is prescribed as
\begin{equation*}\label{eq:ac-manufactured-solution}
u_{\rm ex}(x,y,t)=0.1+0.2\exp(-\beta t)\rho^4(x,y),
\end{equation*}
where $\beta=0.1/\varepsilon$ and $\rho(x,y)=1-\xi^2/a^2-\eta^2/b^2$. The initial condition is taken from the exact solution: $u^0(x,y)=0.1+0.2\rho^4(x,y)$.

For the spatial convergence test, one IMEX step is taken with the discrete-compatible source and $\tau=T=\varepsilon$. The grid is refined by successive doubling from $N=128$ to $N=2048$. Second-order convergence in both norms is observed for $\varepsilon=0.1$, $0.05$, and $0.025$ in Table~\ref{ac-spatial-conver-2d}.

\begin{table}[!htbp]
\centering
\small
\setlength{\tabcolsep}{5pt}
\renewcommand{\arraystretch}{1.1}
\begin{tabular*}{\linewidth}{@{\extracolsep{\fill}}crrrrr@{}}
\hline
$N$ & $E_\infty(h)$ & $\operatorname{ord}_\infty(h)$ & $E_2(h)$ & $\operatorname{ord}_2(h)$ & $\overline{k}_{\rm G}$\\
\hline
\multicolumn{6}{c}{$\varepsilon=0.1$}\\
128 & 3.43e-04 & -- & 4.48e-05 & -- & 10\\
256 & 1.06e-04 & 1.69 & 1.10e-05 & 2.02 & 10\\
512 & 1.95e-05 & 2.45 & 2.25e-06 & 2.29 & 9\\
1024 & 5.07e-06 & 1.94 & 5.45e-07 & 2.05 & 9\\
2048 & 1.10e-06 & 2.21 & 1.42e-07 & 1.94 & 8\\
\hline
\multicolumn{6}{c}{$\varepsilon=0.05$}\\
128 & 1.57e-03 & -- & 1.35e-04 & -- & 11\\
256 & 4.46e-04 & 1.82 & 3.06e-05 & 2.15 & 11\\
512 & 1.04e-04 & 2.10 & 6.48e-06 & 2.24 & 10\\
1024 & 2.35e-05 & 2.15 & 1.48e-06 & 2.13 & 10\\
2048 & 5.01e-06 & 2.23 & 3.95e-07 & 1.90 & 9\\
\hline
\multicolumn{6}{c}{$\varepsilon=0.025$}\\
128 & 7.94e-03 & -- & 4.58e-04 & -- & 14\\
256 & 1.81e-03 & 2.14 & 9.10e-05 & 2.33 & 12\\
512 & 5.08e-04 & 1.83 & 1.96e-05 & 2.21 & 11\\
1024 & 1.07e-04 & 2.24 & 4.25e-06 & 2.21 & 11\\
2048 & 2.30e-05 & 2.22 & 1.14e-06 & 1.90 & 10\\
\hline
\end{tabular*}
\caption{Spatial convergence results for the forced two-dimensional Allen--Cahn equation in Example 1.}
\label{ac-spatial-conver-2d}
\end{table}

For the temporal convergence test with $N=1024$ and $T=\varepsilon$, the time step is halved with $\tau/\varepsilon=1/20,1/40,1/80$, and $1/160$. For each time-step size, two runs using the two source terms are performed and compared at the final time. First-order convergence of the paired indicator is observed in Table~\ref{ac-tem-conver-2d}.
\begin{table}[!htbp]
\centering
\small
\setlength{\tabcolsep}{5pt}
\renewcommand{\arraystretch}{1.1}
\begin{tabular*}{\linewidth}{@{\extracolsep{\fill}}crrrrr@{}}
\hline
$\tau/\varepsilon$ & $E_{\infty,{\rm pair}}(\tau)$ & $\operatorname{ord}_{\infty,{\rm pair}}(\tau)$ & $E_{2,{\rm pair}}(\tau)$ & $\operatorname{ord}_{2,{\rm pair}}(\tau)$ & $\overline{k}_{\rm G}^{\rm cont}/\overline{k}_{\rm G}^{\rm disc}$\\
\hline
\multicolumn{6}{c}{$\varepsilon=0.1$}\\
$1/20$ & 2.34e-03 & -- & 9.10e-04 & -- & 10/10\\
$1/40$ & 1.26e-03 & 0.90 & 4.91e-04 & 0.89 & 11/11\\
$1/80$ & 6.51e-04 & 0.95 & 2.55e-04 & 0.94 & 11/11\\
$1/160$ & 3.32e-04 & 0.97 & 1.30e-04 & 0.97 & 12/12\\
\hline
\multicolumn{6}{c}{$\varepsilon=0.05$}\\
$1/20$ & 2.38e-03 & -- & 9.23e-04 & -- & 11/11\\
$1/40$ & 1.28e-03 & 0.89 & 4.99e-04 & 0.89 & 12/12\\
$1/80$ & 6.67e-04 & 0.94 & 2.60e-04 & 0.94 & 12/12\\
$1/160$ & 3.40e-04 & 0.97 & 1.33e-04 & 0.97 & 13/13\\
\hline
\multicolumn{6}{c}{$\varepsilon=0.025$}\\
$1/20$ & 2.38e-03 & -- & 9.25e-04 & -- & 12/12\\
$1/40$ & 1.28e-03 & 0.89 & 5.01e-04 & 0.89 & 13/13\\
$1/80$ & 6.68e-04 & 0.94 & 2.61e-04 & 0.94 & 14/14\\
$1/160$ & 3.41e-04 & 0.97 & 1.33e-04 & 0.97 & 15/15\\
\hline
\end{tabular*}
\caption{Paired temporal convergence results for the forced two-dimensional Allen--Cahn
equation in Example~1.}
\label{ac-tem-conver-2d}
\end{table}

{\bf Example 2.}
Phase evolution and energy dissipation are examined for the source-free Allen--Cahn equation ($g_{\rm AC}=0$). The rotated five-fold star-shaped domain is defined by
\begin{equation}\label{eq:rotated-star-domain}
\begin{aligned}
\Omega &= \left\{(r\cos\theta,r\sin\theta):
0\leq r<R(\theta),\quad 0\leq\theta<2\pi\right\},\\
R(\theta) &= 0.65+0.20\cos\left(5\left(\theta-\frac{\pi}{10}\right)\right).
\end{aligned}
\end{equation}
The domain is embedded in $\mathcal{B}=[-1,1]^2$, and the initial condition is prescribed as
\begin{equation*}\label{eq:ac-compact-initial}
u^0(x,y)=
\begin{cases}
\displaystyle \frac14\sin^2(4\pi x)\sin^2(4\pi y),
& |x|\leq 0.25,\ |y|\leq 0.25,\\[1mm]
0,&\text{otherwise}.
\end{cases}
\end{equation*}
The phase evolution computed with $N=512$, $\tau=0.01$, $T=1$, and $\varepsilon=0.1$ is shown in Figure~\ref{fig:ac-compact-evolution-2d}. Normalized energy decay is shown in Figure~\ref{fig:ac-normalized-energy-2d} for a $512\times512$ grid with $\tau=0.01$ and $T=1$. Spatial convergence of the total energy is assessed in Figure~\ref{fig:ac-total-energy-convergence-2d}. The energy histories for $N=128,256,512$, and $1024$ are compared with a reference computed at $N_{\rm ref}=2048$, with $\tau=0.05$ and $T=1$ fixed for all grids.
\begin{figure}[!ht]
  \centering
  \includegraphics[width=0.75\linewidth]{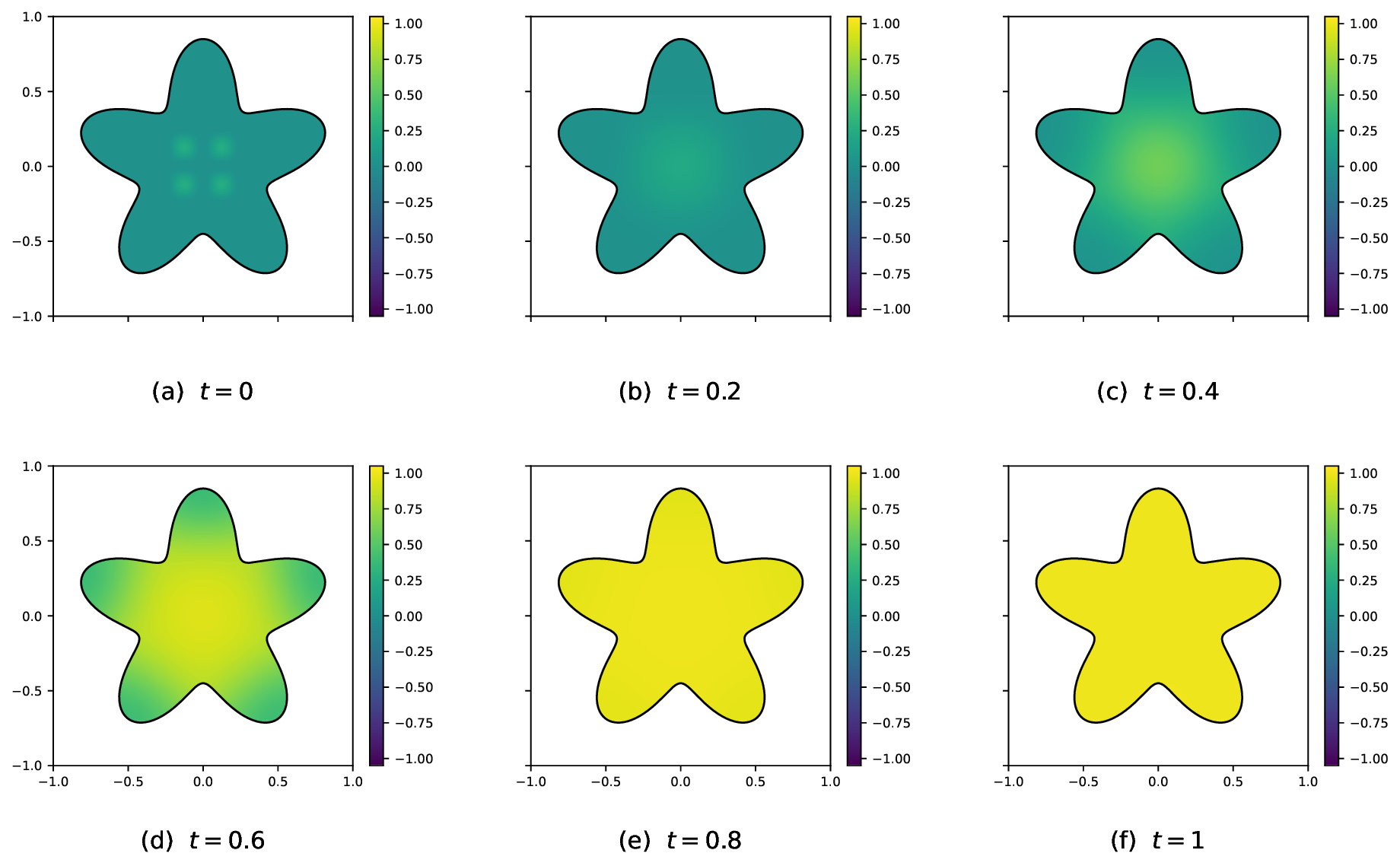}
  \caption{Source-free two-dimensional Allen--Cahn phase evolution in Example 2.}
  \label{fig:ac-compact-evolution-2d}
\end{figure}

\begin{figure}[!ht]
  \centering
  \subfigure[Energy dissipation for different interfacial parameters.\label{fig:ac-normalized-energy-2d}]{
    \includegraphics[hiresbb,trim={-7.920000bp 16.662857bp 23.348571bp 13.371429bp},clip,width=0.46\textwidth,height=0.32\textwidth]{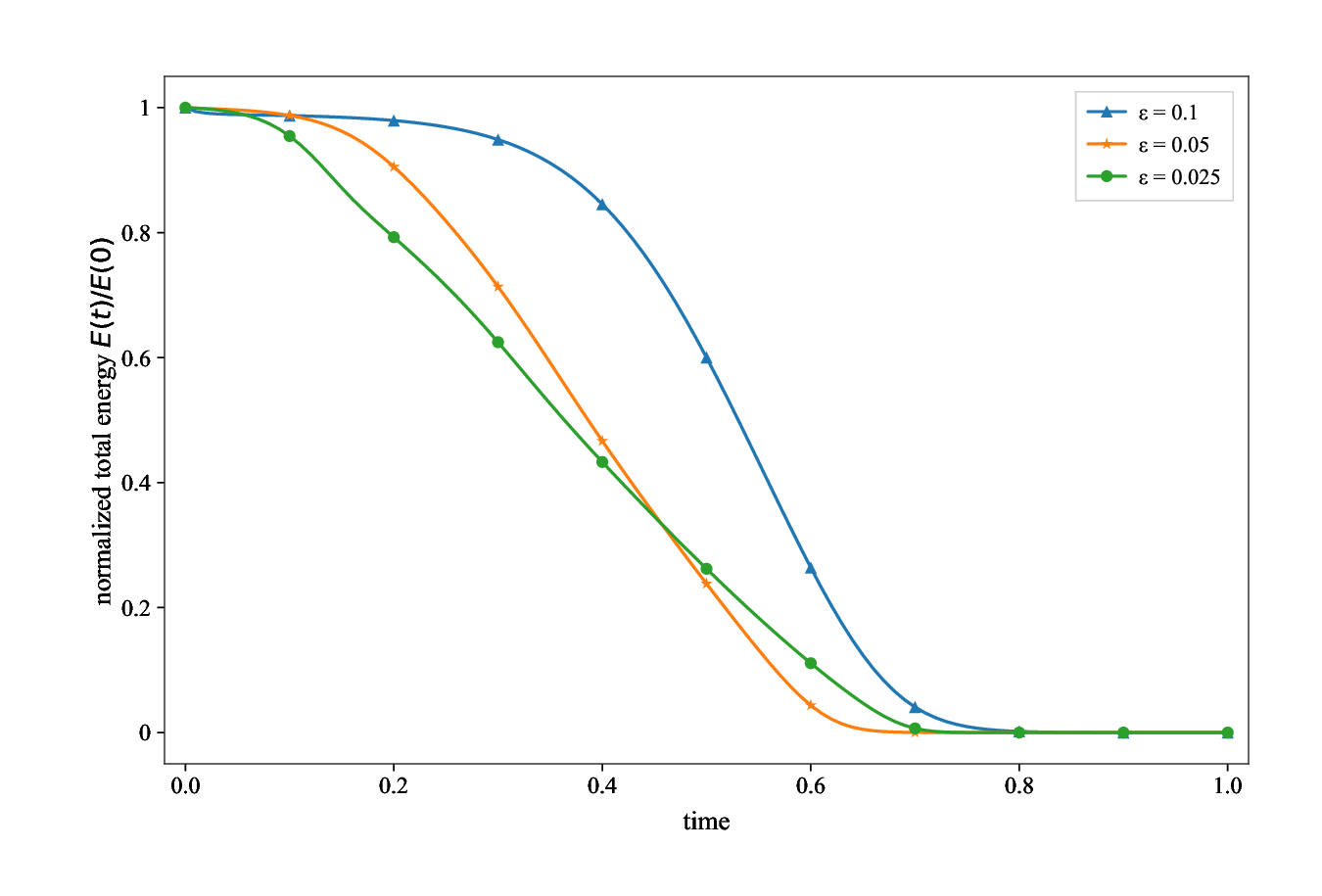}
  }\hspace{0.03\textwidth}%
  \subfigure[Second-order spatial convergence of the total energy (log-log plot).
  \label{fig:ac-total-energy-convergence-2d}]{
    \includegraphics[hiresbb,trim={-9.648763bp 4.733714bp 1.419413bp -9.829143bp},clip,width=0.46\textwidth,height=0.32\textwidth]{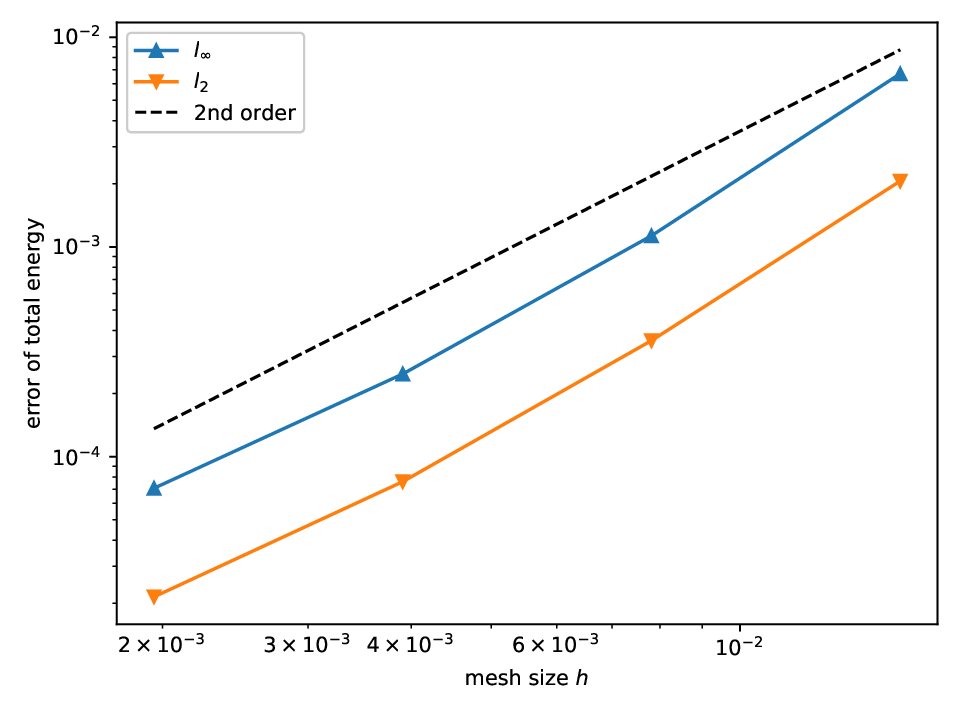}
  }
  \caption{Energy dissipation and spatial convergence of the total
  energy for the source-free two-dimensional Allen--Cahn equation in Example 2.}
  \label{fig:ac-energy-results-2d}
\end{figure}


{\bf Example 3.} 
This example examines the forced three-dimensional Allen--Cahn equation. The ellipsoid
\begin{equation}\label{eq:ac-ellipsoid-3d}
\begin{aligned}
\Omega &= \left\{(x,y,z):
\frac{x^2}{a^2}+\frac{y^2}{b^2}+\frac{z^2}{c^2}<1\right\},\\
a&=0.742,\qquad b=0.558,\qquad c=0.412,
\end{aligned}
\end{equation}
is embedded in $\mathcal{B}=[-1,1]^3$. The exact solution is prescribed as
\begin{equation*}\label{eq:ac-manufactured-solution-3d}
u_{\rm ex}(x,y,z,t)
=0.1+0.2\exp(-\beta t)\rho^4(x,y,z),
\end{equation*}
where $\beta=0.1/\varepsilon$ and $\rho(x,y,z)=1-\frac{x^2}{a^2}-\frac{y^2}{b^2}-\frac{z^2}{c^2}.$ The initial condition is $u^0(x,y,z)=0.1+0.2\rho^4(x,y,z)$.

For the spatial convergence test, the discrete-compatible source is evaluated using the analytic Laplacian of $u_{\rm ex}$. One IMEX step is taken with $\tau=T=\varepsilon$ for $\varepsilon=0.1$, $0.075$, and $0.05$ on grids with $N=64,128,256,512$. Second-order convergence in both norms is observed for all three values of $\varepsilon$ in Table~\ref{ac-spatial-conver-3d}.

\begin{table}[!ht]
\centering
\small
\setlength{\tabcolsep}{5pt}
\renewcommand{\arraystretch}{1.1}
\begin{tabular*}{\linewidth}{@{\extracolsep{\fill}}crrrrr@{}}
\hline
$N$ & $E_\infty(h)$ & $\operatorname{ord}_\infty(h)$ & $E_2(h)$ & $\operatorname{ord}_2(h)$ & $\overline{k}_{\rm G}$\\
\hline
\multicolumn{6}{c}{$\varepsilon=0.1$}\\
64 & 2.76e-03 & -- & 2.06e-04 & -- & 16\\
128 & 5.28e-04 & 2.38 & 4.57e-05 & 2.17 & 17\\
256 & 1.06e-04 & 2.31 & 1.10e-05 & 2.06 & 14\\
512 & 2.58e-05 & 2.04 & 2.71e-06 & 2.02 & 14\\
\hline
\multicolumn{6}{c}{$\varepsilon=0.075$}\\
64 & 5.02e-03 & -- & 3.17e-04 & -- & 17\\
128 & 9.39e-04 & 2.42 & 6.90e-05 & 2.20 & 17\\
256 & 1.94e-04 & 2.27 & 1.64e-05 & 2.07 & 15\\
512 & 4.74e-05 & 2.03 & 4.05e-06 & 2.02 & 15\\
\hline
\multicolumn{6}{c}{$\varepsilon=0.05$}\\
64 & 1.37e-02 & -- & 6.40e-04 & -- & 20\\
128 & 2.13e-03 & 2.68 & 1.29e-04 & 2.31 & 19\\
256 & 4.59e-04 & 2.22 & 3.01e-05 & 2.10 & 16\\
512 & 1.16e-04 & 1.99 & 7.40e-06 & 2.02 & 16\\
\hline
\end{tabular*}
\caption{Spatial convergence results for the forced three-dimensional Allen--Cahn equation in Example~3.}
\label{ac-spatial-conver-3d}
\end{table}

For the temporal convergence test with $N=512$ and $T=\varepsilon$, the time step is halved with $\tau/\varepsilon=1/20,1/40,1/80$, and $1/160$. Each pair of computations is initialized with the same exact values. First-order convergence of the paired indicator is observed in Table~\ref{ac-tem-conver-3d}.

\begin{table}[!ht]
\centering\small
\setlength{\tabcolsep}{4pt}
\renewcommand{\arraystretch}{1.1}
\begin{tabular*}{\linewidth}{@{\extracolsep{\fill}}crrrrr@{}}
\hline
$\tau/\varepsilon$ & $E_{\infty,{\rm pair}}(\tau)$ & $\operatorname{ord}_{\infty,{\rm pair}}(\tau)$ & $E_{2,{\rm pair}}(\tau)$ & $\operatorname{ord}_{2,{\rm pair}}(\tau)$ & $\overline{k}_{\rm G}^{\rm cont}/\overline{k}_{\rm G}^{\rm disc}$\\
\hline
\multicolumn{6}{c}{$\varepsilon=0.1$}\\
$1/20$ & 2.27e-03 & -- & 4.41e-04 & -- & 17/17\\
$1/40$ & 1.21e-03 & 0.91 & 2.37e-04 & 0.90 & 18/18\\
$1/80$ & 6.24e-04 & 0.95 & 1.23e-04 & 0.95 & 19/19\\
$1/160$ & 3.17e-04 & 0.98 & 6.27e-05 & 0.97 & 20/20\\
\hline
\multicolumn{6}{c}{$\varepsilon=0.075$}\\
$1/20$ & 2.34e-03 & -- & 4.52e-04 & -- & 17/17\\
$1/40$ & 1.25e-03 & 0.90 & 2.44e-04 & 0.89 & 18/18\\
$1/80$ & 6.51e-04 & 0.95 & 1.27e-04 & 0.94 & 19/19\\
$1/160$ & 3.32e-04 & 0.97 & 6.47e-05 & 0.97 & 21/21\\
\hline
\multicolumn{6}{c}{$\varepsilon=0.05$}\\
$1/20$ & 2.38e-03 & -- & 4.58e-04 & -- & 19/19\\
$1/40$ & 1.28e-03 & 0.89 & 2.48e-04 & 0.89 & 20/20\\
$1/80$ & 6.64e-04 & 0.94 & 1.29e-04 & 0.94 & 21/21\\
$1/160$ & 3.39e-04 & 0.97 & 6.60e-05 & 0.97 & 24/24\\
\hline
\end{tabular*}
\caption{Paired temporal convergence results for the forced three-dimensional Allen--Cahn equation in Example~3.}
\label{ac-tem-conver-3d}
\end{table}


{\bf Example 4.} 
The source-free Allen--Cahn equation ($g_{\rm AC}=0$) is solved on the three-dimensional toroidal domain
\begin{equation}\label{eq:torus-domain-3d}
\Omega=\left\{(x,y,z):
\left(\sqrt{x^2+y^2}-R\right)^2+z^2<r^2\right\},
\qquad R=0.65,\quad r=0.25.
\end{equation}
The embedding box is $\mathcal B=[-1,1]^3$, and the initial value is defined by 
\begin{equation*}
u^0(x,y,z)=
\begin{cases}
\displaystyle \frac14
\sin^2\!\left(\frac{\pi\xi}{a}\right)
\sin^2\!\left(\frac{\pi z}{a}\right),
& |\xi|\leq a,\ |z|\leq a,\\[1mm]
0,&\text{otherwise},
\end{cases}
\end{equation*}
where  $\xi=\sqrt{x^2+y^2}-R$ and $a=0.15$.

The phase evolution is computed with $N=256$, $\varepsilon=0.1$, $\tau=0.01$, and $T=1$. Horizontal sections $u(x,y,0.075,t)$ are shown in Figure~\ref{fig:ac-compact-evolution-3d}. The section values are interpolated from the three-dimensional snapshots, and a common color range $[0,1.01]$ is used. The initial radial peaks are smoothed, and a nearly uniform positive phase $u=1$ is approached by $T=1$.

\begin{figure}[!ht]
  \centering
  \includegraphics[width=0.75\linewidth]{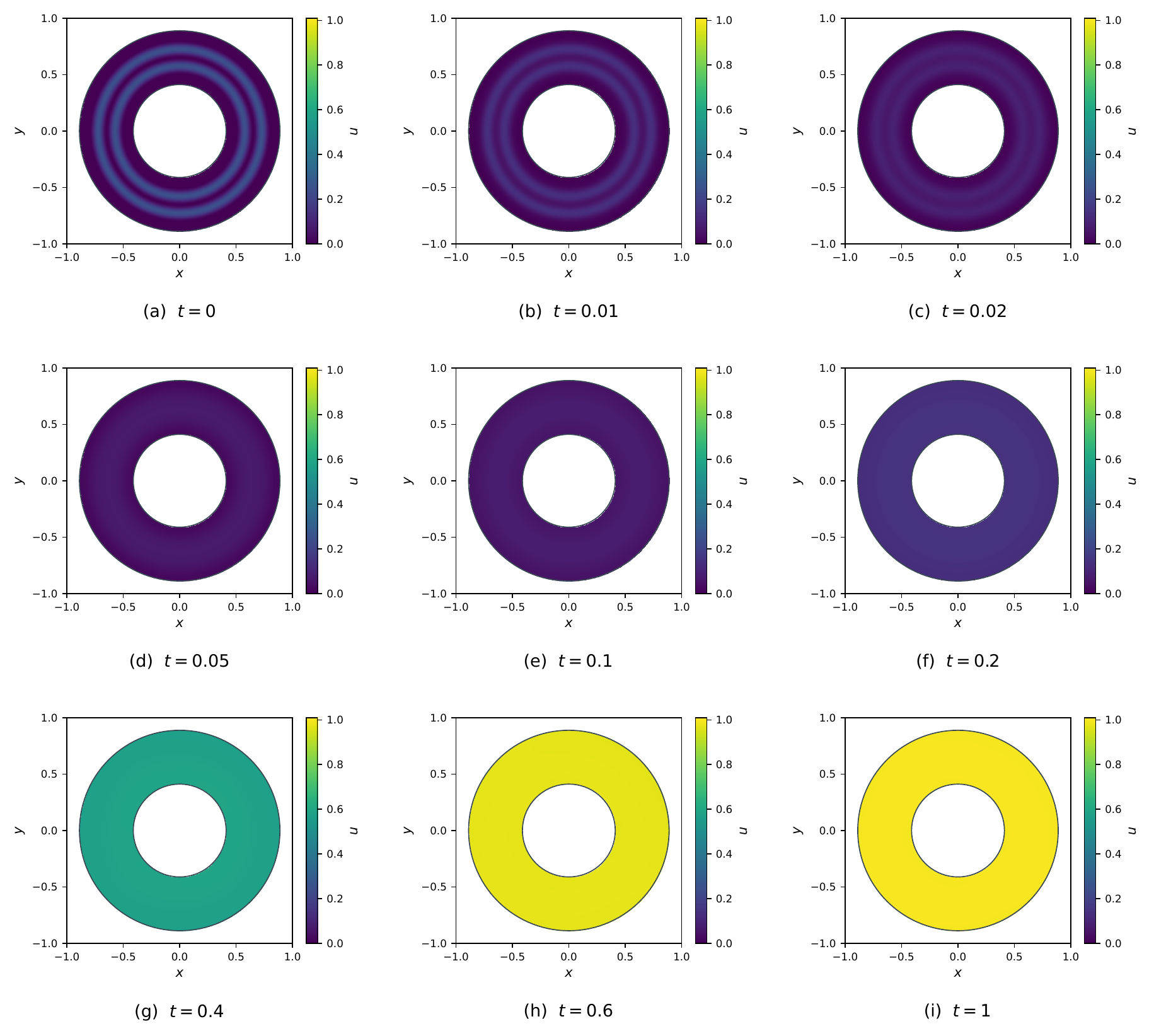}
  \caption{Source-free three-dimensional Allen--Cahn phase evolution in Example~4.}
  \label{fig:ac-compact-evolution-3d}
\end{figure}
 
The dissipation of $E_h(t_n)/E_h(0)$ is shown in Figure~\ref{fig:ac-normalized-energy-3d} for a $512^3$ grid with $\tau=0.01$, $T=1$, and $\varepsilon=0.1$, $0.075$, $0.05$. For spatial convergence of the total energy, the histories for $N=64,128,256$ are compared with a reference computed at $N_{\rm ref}=512$, with $\varepsilon=0.1$, $\tau=0.05$, and $T=1$ fixed. Second-order convergence is observed in Figure~\ref{fig:ac-total-energy-convergence-3d}.

\begin{figure}[!ht]
  \centering
  \subfigure[Energy dissipation for different interfacial parameters.
  \label{fig:ac-normalized-energy-3d}]{
    \includegraphics[hiresbb,trim={-7.920000bp 16.662857bp 23.348571bp 13.371429bp},clip,width=0.46\textwidth,height=0.32\textwidth]{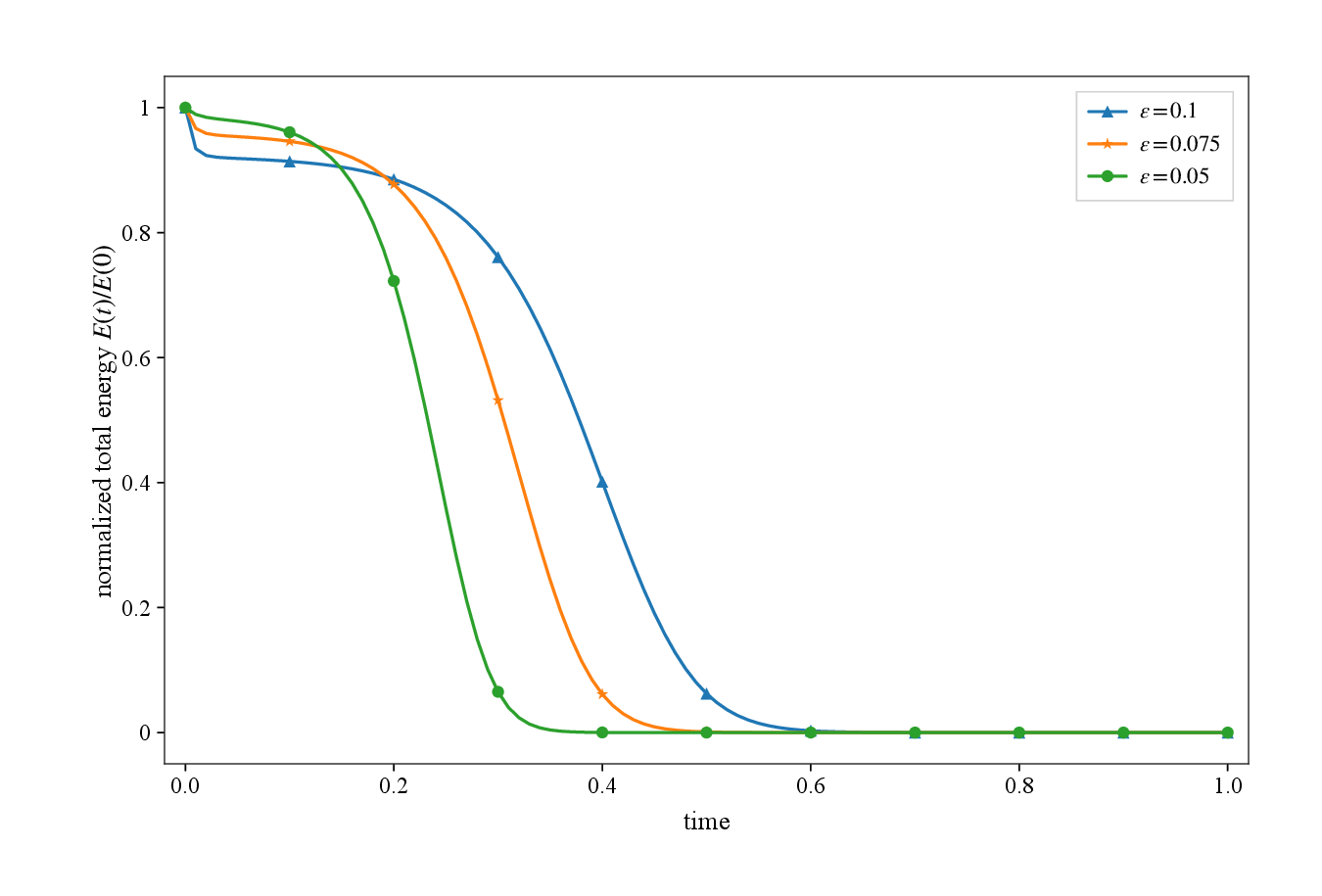}
  }\hspace{0.03\textwidth}%
  \subfigure[Second-order spatial convergence of the total energy (log-log plot).
  \label{fig:ac-total-energy-convergence-3d}]{
    \includegraphics[hiresbb,trim={-8.989007bp 10.223357bp 3.085939bp -7.764790bp},clip,width=0.46\textwidth,height=0.32\textwidth]{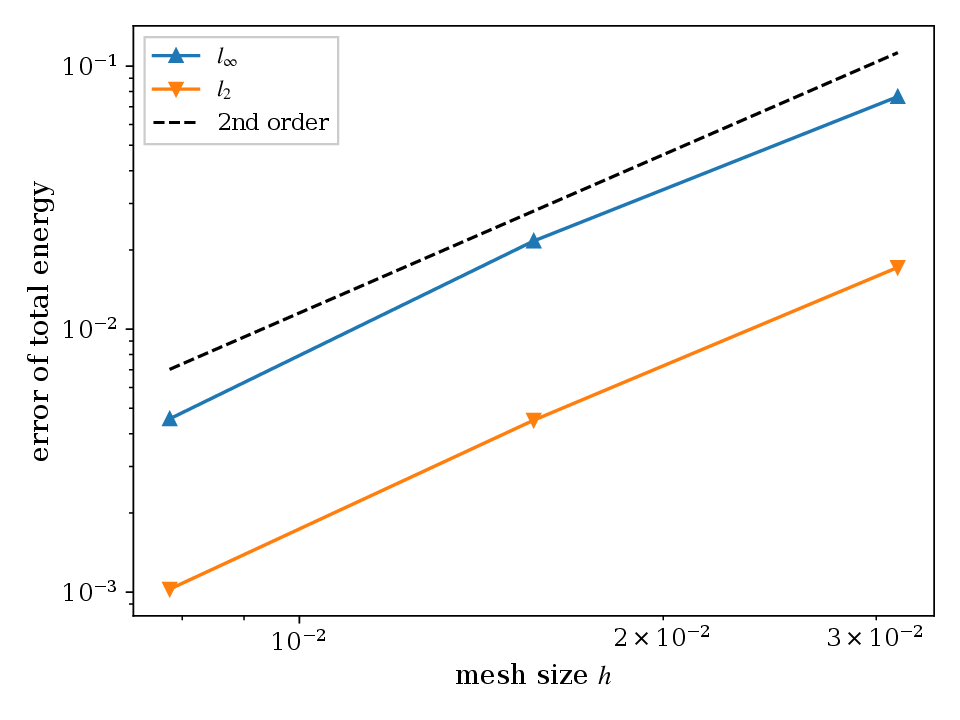}
  }
  \caption{Energy dissipation and spatial convergence of the total
  energy for the source-free three-dimensional Allen--Cahn equation in Example~4.}
  \label{fig:ac-energy-results-3d}
\end{figure}


\subsection{Examples for the Cahn--Hilliard Equation}
The mass-projected Cahn--Hilliard solver is assessed below. Spatial accuracy and convergence of paired temporal-error indicators are tested in Examples~5 and~7. Phase evolution, energy dissipation, mass conservation, and the energy perturbation caused by projection are examined in Examples~6 and~8.

{\bf Example 5.}
This example considers the forced two-dimensional Cahn--Hilliard equation with a manufactured solution. A source $g_{\rm CH}$ is added to the first equation in~\eqref{cahn-hilliard}. The rotated ellipse~\eqref{eq:rotated-ellipse} is embedded in $\mathcal B=[-1,1]^2$, and the exact solution is prescribed as
\begin{equation*}
u_{\rm ex}(x,y,t)
=0.1+0.2\exp(-\beta t)\left(\rho^4(x,y)-\frac15\right),
\end{equation*}
where $\rho(x,y)=1-\frac{\xi^2}{a^2}-\frac{\eta^2}{b^2}$ and $\beta=\frac{0.1}{\varepsilon}$. The initial value is given by $u^0(x,y)=0.1+0.2(\rho^4(x,y)-1/5)$. 

For the spatial convergence test, one IMEX step is taken with the discrete-compatible source and $\tau=T=\varepsilon$. The grids are refined from $N=128$ to $N=2048$ for $\varepsilon=0.1$, $0.05$, and $0.025$. Two positive real shifts are obtained for these parameters. Second-order convergence in both norms is observed in Table~\ref{tab:ch-manufactured-space-2d}.

\begin{table}[!ht]
\centering
\small
\setlength{\tabcolsep}{5pt}
\renewcommand{\arraystretch}{1.1}
\begin{tabular*}{\linewidth}{@{\extracolsep{\fill}}crrrrr@{}}
\hline
$N$ & $E_\infty(h)$ & $\operatorname{ord}_\infty(h)$ & $E_2(h)$ & $\operatorname{ord}_2(h)$ & $\overline{k}_{\rm G}^{\rm s}/\overline{k}_{\rm G}^{\rm l}$\\
\hline
\multicolumn{6}{c}{$\varepsilon=0.1$}\\
128 & 2.43e-04 & -- & 8.20e-05 & -- & 12/10\\
256 & 6.70e-05 & 1.86 & 2.23e-05 & 1.88 & 12/9\\
512 & 1.62e-05 & 2.05 & 5.21e-06 & 2.10 & 12/9\\
1024 & 4.06e-06 & 2.00 & 1.36e-06 & 1.94 & 11/8\\
2048 & 1.01e-06 & 2.01 & 3.44e-07 & 1.98 & 10/7\\
\hline
\multicolumn{6}{c}{$\varepsilon=0.05$}\\
128 & 6.29e-04 & -- & 7.07e-05 & -- & 12/11\\
256 & 1.82e-04 & 1.79 & 1.66e-05 & 2.09 & 12/10\\
512 & 4.08e-05 & 2.16 & 3.65e-06 & 2.19 & 12/10\\
1024 & 9.29e-06 & 2.14 & 8.78e-07 & 2.06 & 11/9\\
2048 & 1.96e-06 & 2.24 & 2.28e-07 & 1.95 & 10/8\\
\hline
\multicolumn{6}{c}{$\varepsilon=0.025$}\\
128 & 3.10e-03 & -- & 2.02e-04 & -- & 12/13\\
256 & 7.67e-04 & 2.01 & 4.23e-05 & 2.25 & 12/11\\
512 & 2.05e-04 & 1.90 & 9.17e-06 & 2.21 & 12/11\\
1024 & 4.38e-05 & 2.23 & 2.03e-06 & 2.18 & 11/10\\
2048 & 9.28e-06 & 2.24 & 5.41e-07 & 1.91 & 10/9\\
\hline
\end{tabular*}
\caption{Spatial convergence results for the forced two-dimensional Cahn--Hilliard
equation in Example 5.}
\label{tab:ch-manufactured-space-2d}
\end{table}

For the temporal convergence test, $N=1024$ and $T=1.25\varepsilon^3$ are fixed, and $\tau=\frac{T}{16},\frac{T}{32},\frac{T}{64},\frac{T}{128}$ are used. Since $\tau<\varepsilon^3$ and $s=2$, complex-conjugate shifts are obtained in all cases. First-order convergence of the paired indicator is observed in Table~\ref{tab:ch-manufactured-time-2d}.

\begin{table}[!ht]
\centering
\small
\setlength{\tabcolsep}{5pt}
\renewcommand{\arraystretch}{1.1}
\begin{tabular*}{\linewidth}{@{\extracolsep{\fill}}crrrrr@{}}
\hline
$\tau/T$ & $E_{\infty,{\rm pair}}(\tau)$ & $\operatorname{ord}_{\infty,{\rm pair}}(\tau)$ & $E_{2,{\rm pair}}(\tau)$ & $\operatorname{ord}_{2,{\rm pair}}(\tau)$ & $\overline{k}_{\rm G}^{\rm cont}/\overline{k}_{\rm G}^{\rm disc}$\\
\hline
\multicolumn{6}{c}{$\varepsilon=0.1$}\\
$1/16$ & 1.94e-05 & -- & 6.66e-06 & -- & 9/9\\
$1/32$ & 1.03e-05 & 0.92 & 3.52e-06 & 0.92 & 9/9\\
$1/64$ & 5.29e-06 & 0.96 & 1.81e-06 & 0.96 & 9/9\\
$1/128$ & 2.69e-06 & 0.98 & 9.17e-07 & 0.98 & 9/9\\
\hline
\multicolumn{6}{c}{$\varepsilon=0.05$}\\
$1/16$ & 1.21e-06 & -- & 4.27e-07 & -- & 9/9\\
$1/32$ & 6.14e-07 & 0.98 & 2.19e-07 & 0.96 & 9.03/9.03\\
$1/64$ & 3.09e-07 & 0.99 & 1.11e-07 & 0.98 & 10/10\\
$1/128$ & 1.55e-07 & 0.99 & 5.61e-08 & 0.99 & 10/10\\
\hline
\multicolumn{6}{c}{$\varepsilon=0.025$}\\
$1/16$ & 7.00e-08 & -- & 2.44e-08 & -- & 10/10\\
$1/32$ & 3.51e-08 & 0.99 & 1.23e-08 & 0.99 & 10/10\\
$1/64$ & 1.76e-08 & 1.00 & 6.17e-09 & 1.00 & 10/10\\
$1/128$ & 8.81e-09 & 1.00 & 3.09e-09 & 1.00 & 11/11\\
\hline
\end{tabular*}
\caption{Paired temporal convergence results for the forced two-dimensional Cahn--Hilliard
problem in Example 5.}
\label{tab:ch-manufactured-time-2d}
\end{table}


{\bf Example 6.} 
The source-free Cahn--Hilliard equation ($g_{\rm CH}=0$) is solved on the rotated five-fold star-shaped domain~\eqref{eq:rotated-star-domain} in $\mathcal B=[-1,1]^2$. The initial condition is prescribed as
\begin{equation*}\label{eq:ch-compact-initial}
u^0(x,y)=
\begin{cases}
\displaystyle \frac14\sin^4(4\pi x)\sin^4(4\pi y),
& |x|\leq0.25,\ |y|\leq0.25,\\[1mm]
0,&\text{otherwise}.
\end{cases}
\end{equation*}

The phase evolution computed on a $512\times512$ grid with $\varepsilon=0.1$, $\tau=0.00125$, and $T=1$ is shown in Figure~\ref{fig:combined-ch-evolution-2d}.

\begin{figure}[!ht]
  \centering
  \includegraphics[width=0.75\linewidth]{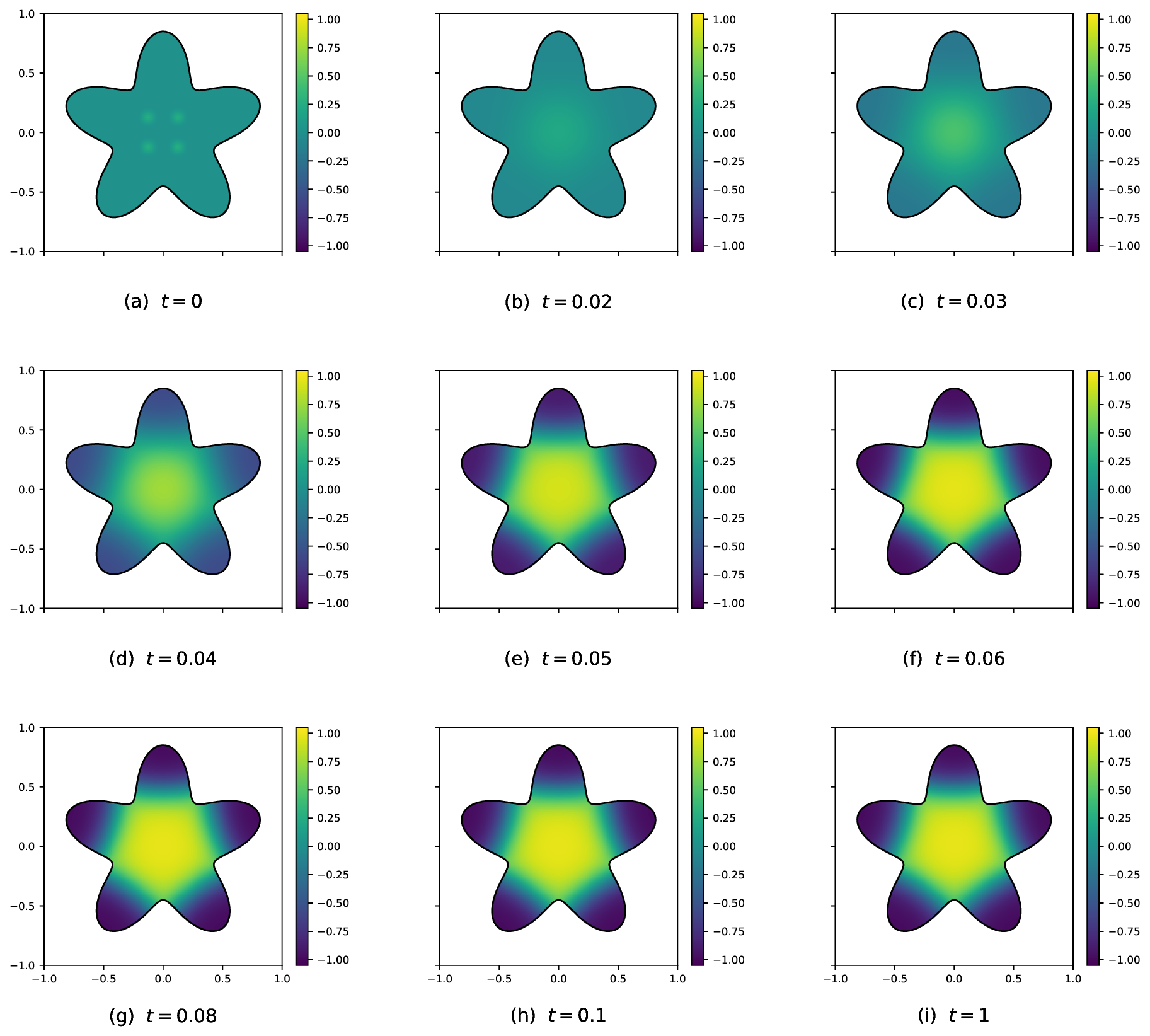}
  \caption{Source-free two-dimensional Cahn--Hilliard phase evolution in Example 6.}
  \label{fig:combined-ch-evolution-2d}
\end{figure}

Energy dissipation is observed in Figure~\ref{fig:combined-ch-normalized-energy-2d} for $N=512$, $\tau=0.005$, $T=1$, and $\varepsilon=0.1$, $0.05$, $0.025$. For spatial convergence of the total energy, $\varepsilon=0.1$, $\tau=0.001$, and $T=0.02$ are fixed. The results for $N=128,256,512,1024$ are compared with a reference at $N_{\rm ref}=2048$. Second-order convergence is observed in Figure~\ref{fig:combined-ch-energy-convergence-2d}.

Mass conservation is tested separately with $N=512$, $\varepsilon=0.1$, $\tau=0.00125$, and $T=1$. Two otherwise identical runs are performed with and without projection. The relative mass errors after projection, before each projection, and in the independent unprojected run are compared in Figure~\ref{fig:combined-ch-mass-conservation}. Maximum errors of $2.940857\times10^{-13}$, $1.223756\times10^{-3}$, and $9.308733\times10^{-1}$ are obtained, respectively. The prescribed mass is thus restored to near machine precision by projection at every step.

The energy perturbation caused by projection is measured in the energy-dissipation runs ($N=512$, $\tau=0.005$, $T=1$). The single-step ratios $r_{\mathrm{proj}}^{n+1}$ are shown in Figure~\ref{fig:combined-ch-projection-energy}; a maximum of $6.903\times10^{-4}\%$ is attained at $\varepsilon=0.05$. Cumulative ratios $R_{\mathrm{proj}}$ of $4.920\times10^{-2}\%$, $3.847\times10^{-2}\%$, and $2.793\times10^{-2}\%$ are obtained for $\varepsilon=0.1$, $0.05$, and $0.025$, respectively. Only a small energy perturbation relative to the observed dissipation is introduced in these tests.

\begin{figure}[!ht]
  \centering
  \makebox[\textwidth][c]{%
    \subfigure[Energy dissipation for different interfacial parameters.\label{fig:combined-ch-normalized-energy-2d}]{
      \includegraphics[hiresbb,trim={-7.920000bp 16.662857bp 23.348571bp 13.371429bp},clip,width=0.35\textwidth,height=0.245\textwidth]{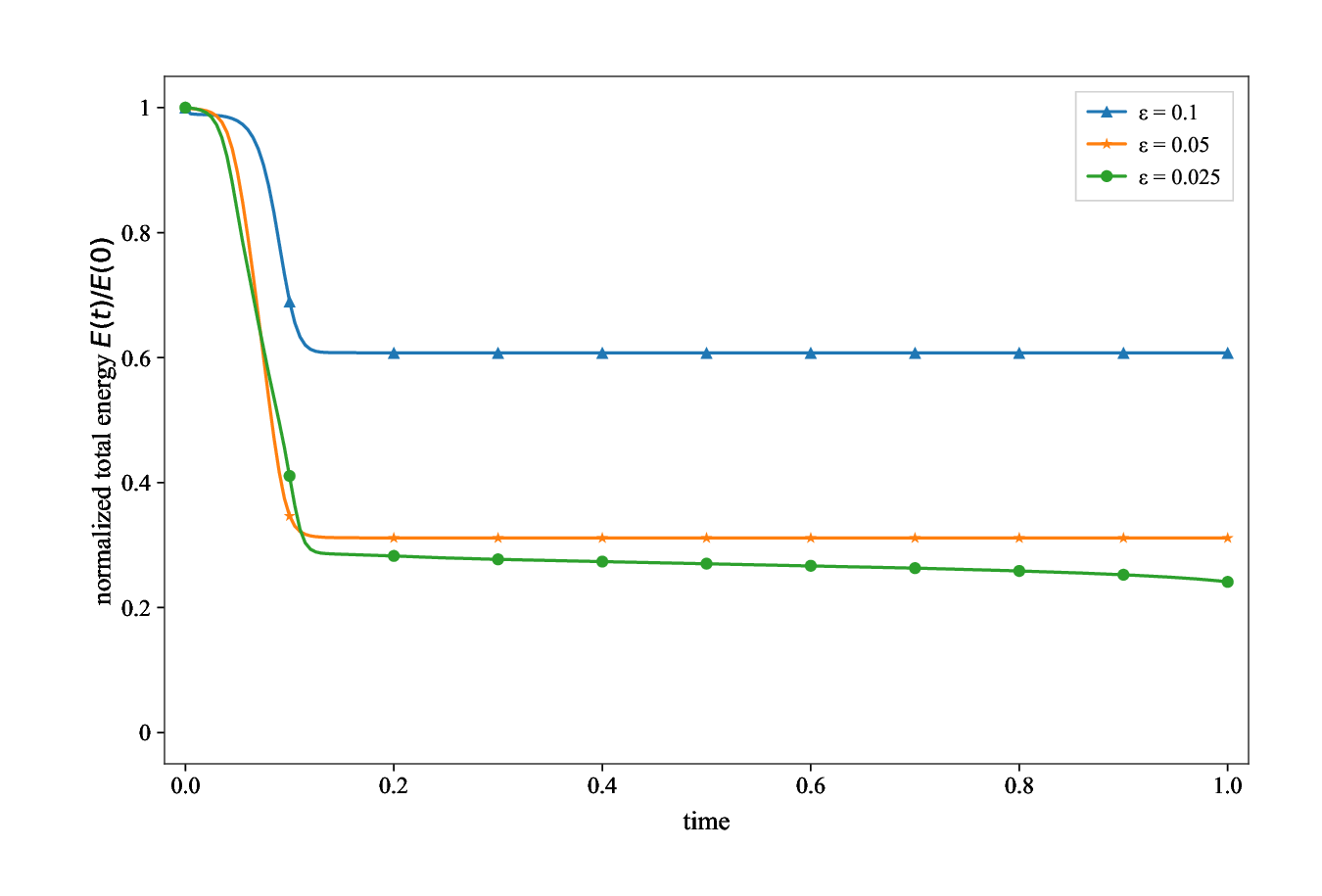}
    }\hspace{0.04\textwidth}%
    \subfigure[Second-order spatial convergence of the total energy (log-log plot).\label{fig:combined-ch-energy-convergence-2d}]{
      \includegraphics[hiresbb,trim={-9.648763bp 4.733714bp 1.419413bp -9.829143bp},clip,width=0.35\textwidth,height=0.245\textwidth]{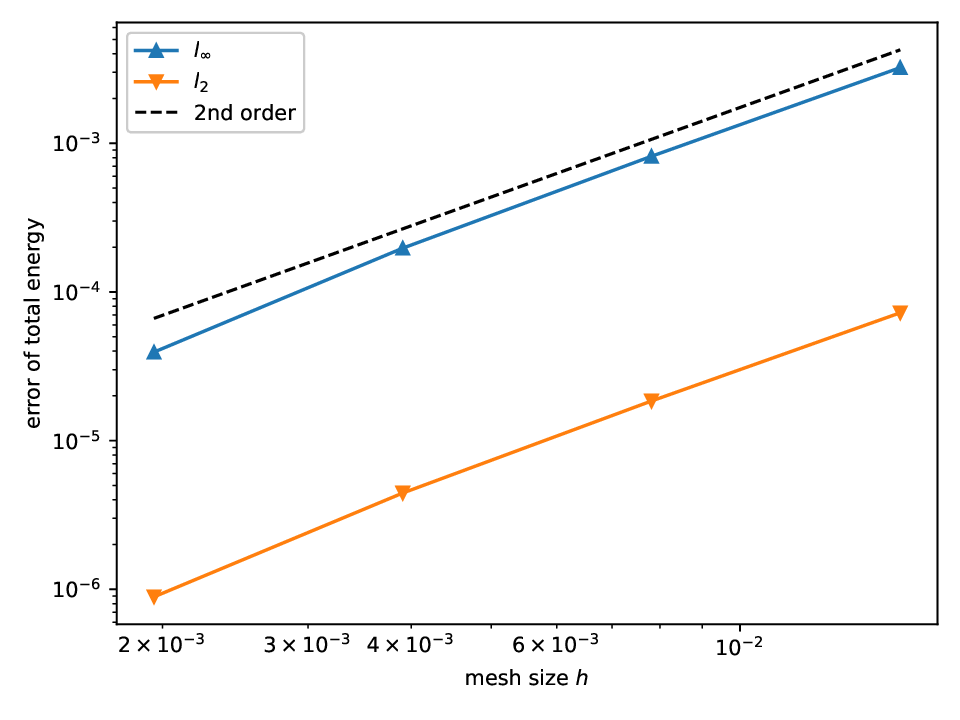}
    }%
  }

  \par\bigskip

  \makebox[\textwidth][c]{%
    \subfigure[Relative mass errors after projection, before projection, and without projection.\label{fig:combined-ch-mass-conservation}]{
      \includegraphics[hiresbb,trim={-7.272000bp 8.310857bp 17.660571bp -6.752571bp},clip,width=0.35\textwidth,height=0.245\textwidth]{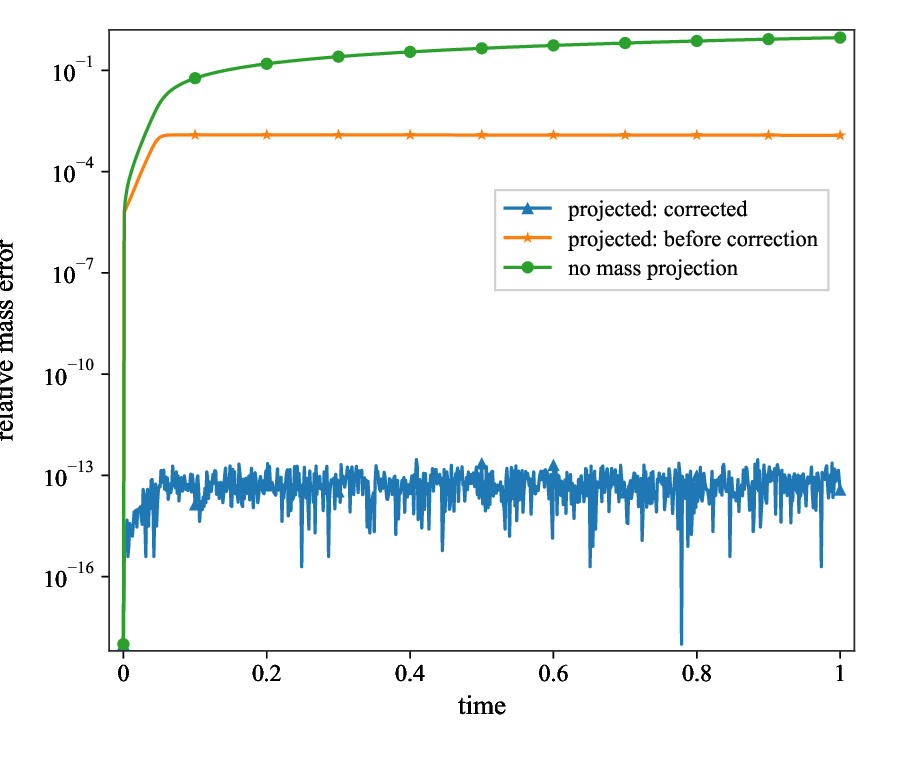}
    }\hspace{0.04\textwidth}%
    \subfigure[Relative energy perturbation induced by the mass projection.\label{fig:combined-ch-projection-energy}]{
      \includegraphics[hiresbb,trim={-7.920000bp 16.662857bp 23.348571bp 13.371429bp},clip,width=0.35\textwidth,height=0.245\textwidth]{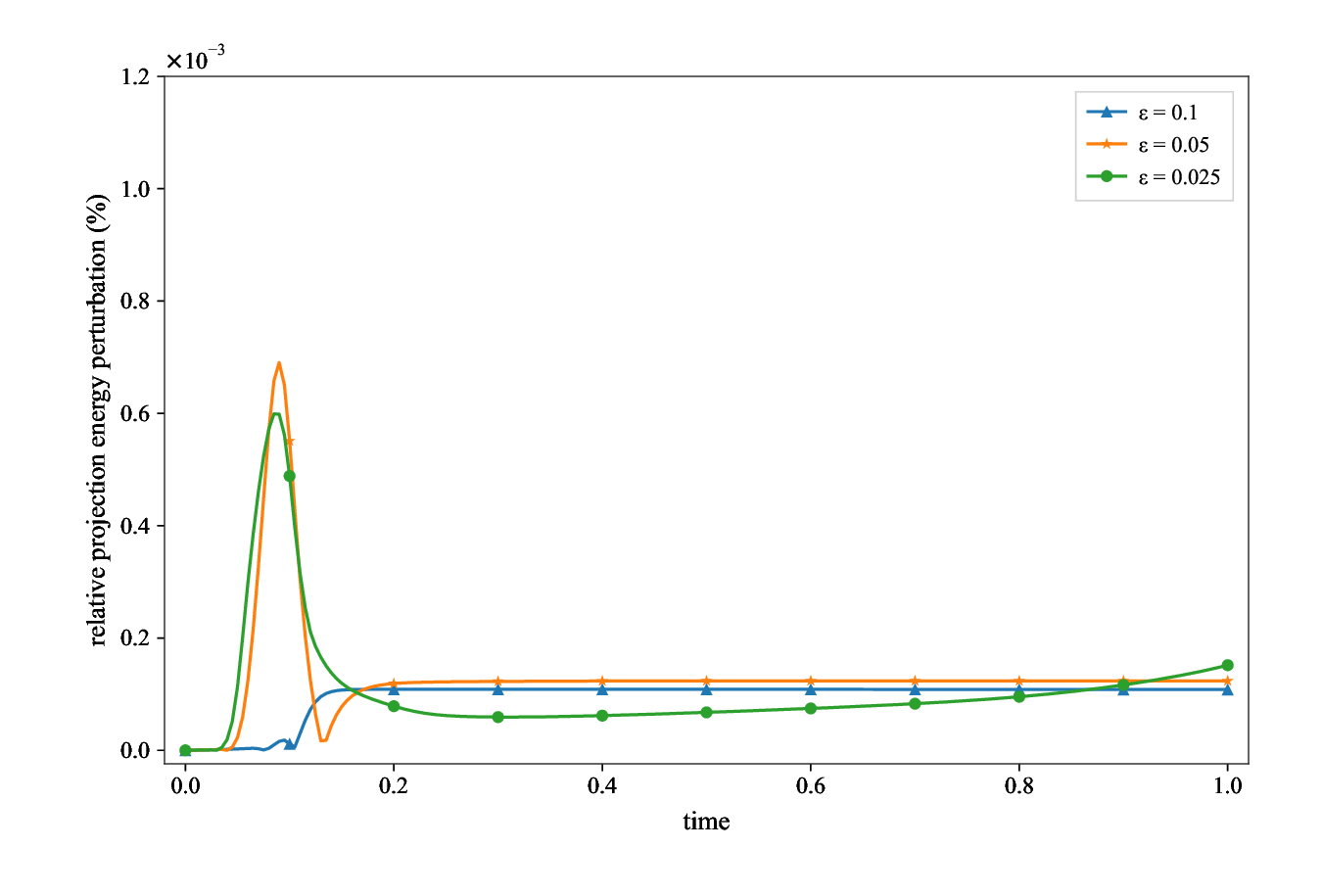}
    }%
  }
\caption{Energy behavior and mass conservation for the source-free two-dimensional
Cahn--Hilliard equation in Example~6.}
  \label{fig:combined-ch-energy-mass}
\end{figure}

{\bf Example 7.} 
This example considers the forced three-dimensional Cahn--Hilliard equation. The ellipsoid~\eqref{eq:ac-ellipsoid-3d} is embedded in $\mathcal{B}=[-1,1]^3$, and the exact solution is prescribed as
\begin{equation*}\label{eq:ch-manufactured-solution-3d}
u_{\rm ex}(x,y,z,t)
=0.8+0.2\exp(-\beta t)
\left(\rho^4(x,y,z)-\frac{128}{1155}\right),
\end{equation*}
where $\beta=0.1/\varepsilon$ and $\rho(x,y,z)=1-\frac{x^2}{a^2}-\frac{y^2}{b^2}-\frac{z^2}{c^2}.$ The initial condition is $u^0(x,y,z)=0.8+0.2(\rho^4(x,y,z)-128/1155)$.

For the spatial convergence test, one IMEX step is taken with $\tau=T=\varepsilon$ for $\varepsilon=0.1$, $0.075$, and $0.05$ on grids with $N=64,128,256,512$. Two positive real shifts are obtained. Second-order convergence is observed for all three values of $\varepsilon$ in Table~\ref{tab:ch-manufactured-space-3d}.

\begin{table}[!ht]
\centering
\small
\setlength{\tabcolsep}{5pt}
\renewcommand{\arraystretch}{1.1}
\begin{tabular*}{\linewidth}{@{\extracolsep{\fill}}crrrrr@{}}
\hline
$N$ & $E_\infty(h)$ & $\operatorname{ord}_\infty(h)$ & $E_2(h)$ & $\operatorname{ord}_2(h)$ & $\overline{k}_{\rm G}^{\rm s}/\overline{k}_{\rm G}^{\rm l}$\\
\hline
\multicolumn{6}{c}{$\varepsilon=0.1$}\\
64 & 1.31e-02 & -- & 1.09e-03 & -- & 21/17\\
128 & 2.30e-03 & 2.51 & 2.27e-04 & 2.27 & 23/18\\
256 & 4.77e-04 & 2.27 & 5.19e-05 & 2.13 & 20/15\\
512 & 1.12e-04 & 2.09 & 1.29e-05 & 2.01 & 20/16\\
\hline
\multicolumn{6}{c}{$\varepsilon=0.075$}\\
64 & 2.50e-02 & -- & 1.64e-03 & -- & 21/18\\
128 & 4.52e-03 & 2.47 & 3.45e-04 & 2.25 & 23/19\\
256 & 9.05e-04 & 2.32 & 8.05e-05 & 2.10 & 20/16\\
512 & 2.20e-04 & 2.04 & 1.98e-05 & 2.02 & 20/16\\
\hline
\multicolumn{6}{c}{$\varepsilon=0.05$}\\
64 & 6.25e-02 & -- & 3.28e-03 & -- & 21/20\\
128 & 1.07e-02 & 2.55 & 6.81e-04 & 2.27 & 23/20\\
256 & 2.19e-03 & 2.28 & 1.59e-04 & 2.10 & 20/17\\
512 & 5.54e-04 & 1.99 & 3.91e-05 & 2.02 & 20/17\\
\hline
\end{tabular*}
\caption{Spatial convergence results for the forced three-dimensional Cahn--Hilliard equation in Example~7.}
\label{tab:ch-manufactured-space-3d}
\end{table}

For the temporal convergence test, $N=256$ and $T=\varepsilon$ are fixed, and $\tau/\varepsilon=1/10,1/20,1/40$, and $1/80$ are used. Two positive real shifts are obtained. First-order convergence of the paired indicator is observed for all three values of $\varepsilon$ in Table~\ref{tab:ch-manufactured-time-3d}.

\begin{table}[!ht]
\centering
\small
\setlength{\tabcolsep}{5pt}
\renewcommand{\arraystretch}{1.1}
\begin{tabular*}{\linewidth}{@{\extracolsep{\fill}}crrrrr@{}}
\hline
$\tau/\varepsilon$ & $E_{\infty,{\rm pair}}(\tau)$ & $\operatorname{ord}_{\infty,{\rm pair}}(\tau)$ & $E_{2,{\rm pair}}(\tau)$ & $\operatorname{ord}_{2,{\rm pair}}(\tau)$ & $\overline{k}_{\rm G}^{\rm cont}/\overline{k}_{\rm G}^{\rm disc}$\\
\hline
\multicolumn{6}{c}{$\varepsilon=0.1$}\\
$1/10$ & 1.93e-03 & -- & 3.94e-04 & -- & 33/33\\
$1/20$ & 9.65e-04 & 1.00 & 1.97e-04 & 1.00 & 32/32\\
$1/40$ & 4.82e-04 & 1.00 & 9.82e-05 & 1.00 & 32/32\\
$1/80$ & 2.41e-04 & 1.00 & 4.91e-05 & 1.00 & 33/33\\
\hline
\multicolumn{6}{c}{$\varepsilon=0.075$}\\
$1/10$ & 1.95e-03 & -- & 4.15e-04 & -- & 34/34\\
$1/20$ & 9.73e-04 & 1.00 & 2.07e-04 & 1.00 & 33/33\\
$1/40$ & 4.86e-04 & 1.00 & 1.03e-04 & 1.00 & 33/33\\
$1/80$ & 2.43e-04 & 1.00 & 5.16e-05 & 1.00 & 34/34\\
\hline
\multicolumn{6}{c}{$\varepsilon=0.05$}\\
$1/10$ & 1.93e-03 & -- & 4.37e-04 & -- & 35/35\\
$1/20$ & 9.66e-04 & 1.00 & 2.18e-04 & 1.00 & 35/35\\
$1/40$ & 4.83e-04 & 1.00 & 1.09e-04 & 1.00 & 34.98/34.98\\
$1/80$ & 2.41e-04 & 1.00 & 5.43e-05 & 1.00 & 35/35\\
\hline
\end{tabular*}
\caption{Paired temporal convergence results for the forced three-dimensional Cahn--Hilliard equation in Example~7.}
\label{tab:ch-manufactured-time-3d}
\end{table}

\noindent\textbf{Accuracy and GMRES iteration counts for complex-conjugate shifts}

We compare an implementation using one complex solve (C1) with an implementation using two successive complex solves (C2). C1 recovers the real-valued solution using \eqref{eq:ch-complex-reconstruction}, whereas C2 solves the two shifted Neumann problems consecutively and takes the real part of the resulting solution.

The ellipsoidal domain and manufactured solution are those of Example~7. We take one IMEX step with $\varepsilon=0.1$, $s=2$, and $\tau=T=5\times10^{-4}$, which gives the complex-conjugate shifts $100\pm100\mathrm{i}$. Table~\ref{tab:complex-shift-comparison} reports the errors and total GMRES iteration count $k_{\rm G}$ per time step. For C2, $k_{\rm G}$ is the sum of the iteration counts for the two shifted solves.
\begin{table}[!htbp]
\centering
\small
\setlength{\tabcolsep}{3pt}
\renewcommand{\arraystretch}{1.1}
\begin{tabular*}{\linewidth}
{@{\extracolsep{\fill}}ccrrrrr@{}}
\hline
$N$ & Method & $E_2(h)$ & $\operatorname{ord}_2(h)$ & $E_\infty(h)$
    & $\operatorname{ord}_\infty(h)$ & $k_{\rm G}$\\
\hline
64  & C1 & 1.33e-03 & --   & 1.58e-02 & --   & 15\\
    & C2 & 8.58e-04 & --   & 8.34e-03 & --   & 30\\
\hline
128 & C1 & 2.93e-04 & 2.19 & 3.05e-03 & 2.37 & 16\\
    & C2 & 1.91e-04 & 2.16 & 1.39e-03 & 2.58 & 32\\
\hline
256 & C1 & 6.96e-05 & 2.07 & 5.82e-04 & 2.39 & 14\\
    & C2 & 4.65e-05 & 2.04 & 3.05e-04 & 2.19 & 28\\
\hline
512 & C1 & 1.72e-05 & 2.02 & 1.51e-04 & 1.95 & 14\\
    & C2 & 1.15e-05 & 2.01 & 7.67e-05 & 1.99 & 28\\
\hline
\end{tabular*}
\caption{Accuracy and GMRES iteration counts for one and
two complex solves.}
\label{tab:complex-shift-comparison}
\end{table}

In this test, C1 retains the observed second-order spatial convergence in both norms, with errors of the same order of magnitude as those of C2, while halving the total GMRES iteration count per time step.


{\bf Example 8.} 
The source-free three-dimensional Cahn--Hilliard equation ($g_{\rm CH}=0$) is solved on the toroidal domain~\eqref{eq:torus-domain-3d} in $\mathcal{B}=[-1,1]^3$. The initial condition is prescribed as
\begin{equation*}\label{eq:ch-compact-initial-3d}
u^0(x,y,z)=
\begin{cases}
\displaystyle
\frac14\sin^4\!\left(\frac{\pi(x-R)}{0.25}\right)
\sin^4\!\left(\frac{\pi y}{0.25}\right)
\cos^4\!\left(\frac{\pi z}{0.5}\right),
& |x-R|\leq0.25,\ |y|\leq0.25,\ |z|\leq0.25,\\
0,&\text{otherwise}.
\end{cases}
\end{equation*}

The phase evolution is computed on a $256^3$ grid with $\varepsilon=0.1$, $\tau=0.00125$, and $T=1$. The cross-section at $z=0$ is shown in Figure~\ref{fig:combined-ch-evolution-3d}.

\begin{figure}[!ht]
  \centering
  \includegraphics[width=0.75\linewidth]{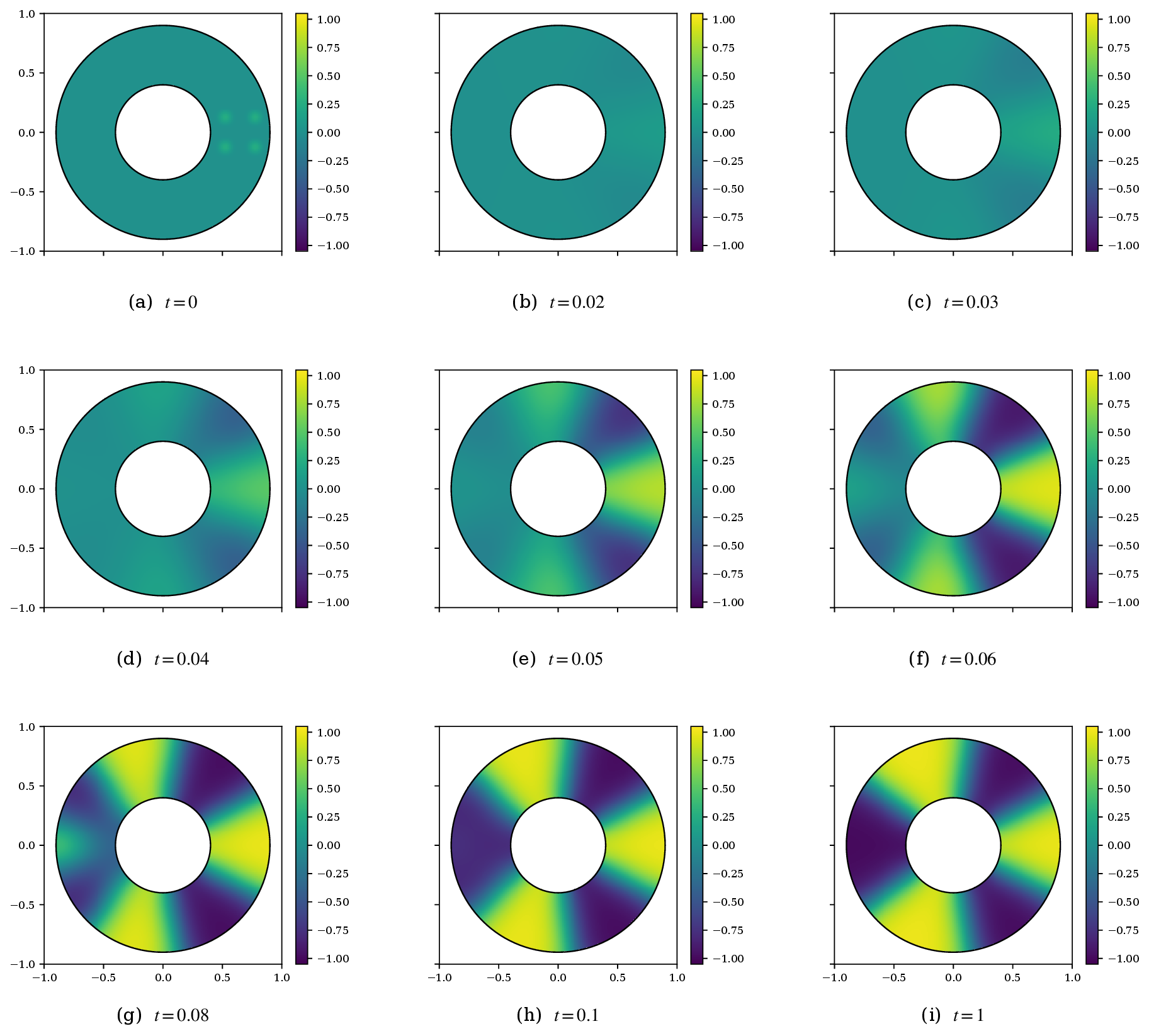}
  \caption{Source-free three-dimensional Cahn--Hilliard phase evolution in Example 8.}
  \label{fig:combined-ch-evolution-3d}
\end{figure}

Normalized energy dissipation is observed in Figure~\ref{fig:combined-ch-normalized-energy-3d} on a $512^3$ grid with $\tau=0.01$, $T=1$, and $\varepsilon=0.05$, $0.075$, $0.1$. To assess the spatial convergence of the total energy, $\varepsilon=0.1$, $\tau=0.001$, and $T=0.02$ are fixed. The results for $N=64,128,256$ are compared with a reference at $N_{\rm ref}=512$. Second-order convergence is observed relative to this reference in Figure~\ref{fig:combined-ch-energy-convergence-3d}.

Mass conservation is assessed in the same energy-dissipation runs. For $\varepsilon=0.05$, $0.075$, and $0.1$, maximum relative mass errors of $6.581197\times10^{-11}$, $5.799868\times10^{-11}$, and $3.749796\times10^{-11}$ are obtained after projection in Figure~\ref{fig:combined-ch-mass-conservation-3d}. Before projection, the corresponding maxima are $5.448396\times10^{-3}$, $3.297774\times10^{-3}$, and $1.896395\times10^{-3}$. The mass defects are reduced by approximately eight orders of magnitude, and a relative mass error below $7\times10^{-11}$ is maintained.

The single-step energy perturbations are shown in Figure~\ref{fig:combined-ch-projection-energy-3d}. For $\varepsilon=0.05$, $0.075$, and $0.1$, maximum ratios $r_{\mathrm{proj}}^{n+1}$ of $7.577413\times10^{-4}\%$, $3.232565\times10^{-4}\%$, and $1.351845\times10^{-4}\%$ are obtained. The corresponding cumulative ratios $R_{\mathrm{proj}}$ are $1.098217\times10^{-2}\%$, $5.531078\times10^{-3}\%$, and $1.835807\times10^{-3}\%$. Only a small perturbation relative to the observed dissipation is introduced, and no increase in the projected total energy is detected in these runs.

\begin{figure}[!ht]
  \centering
  \makebox[\textwidth][c]{%
    \subfigure[Energy dissipation for different interfacial parameters.\label{fig:combined-ch-normalized-energy-3d}]{
      \includegraphics[hiresbb,trim={-7.920000bp 16.662857bp 23.348571bp 13.371429bp},clip,width=0.35\textwidth,height=0.245\textwidth]{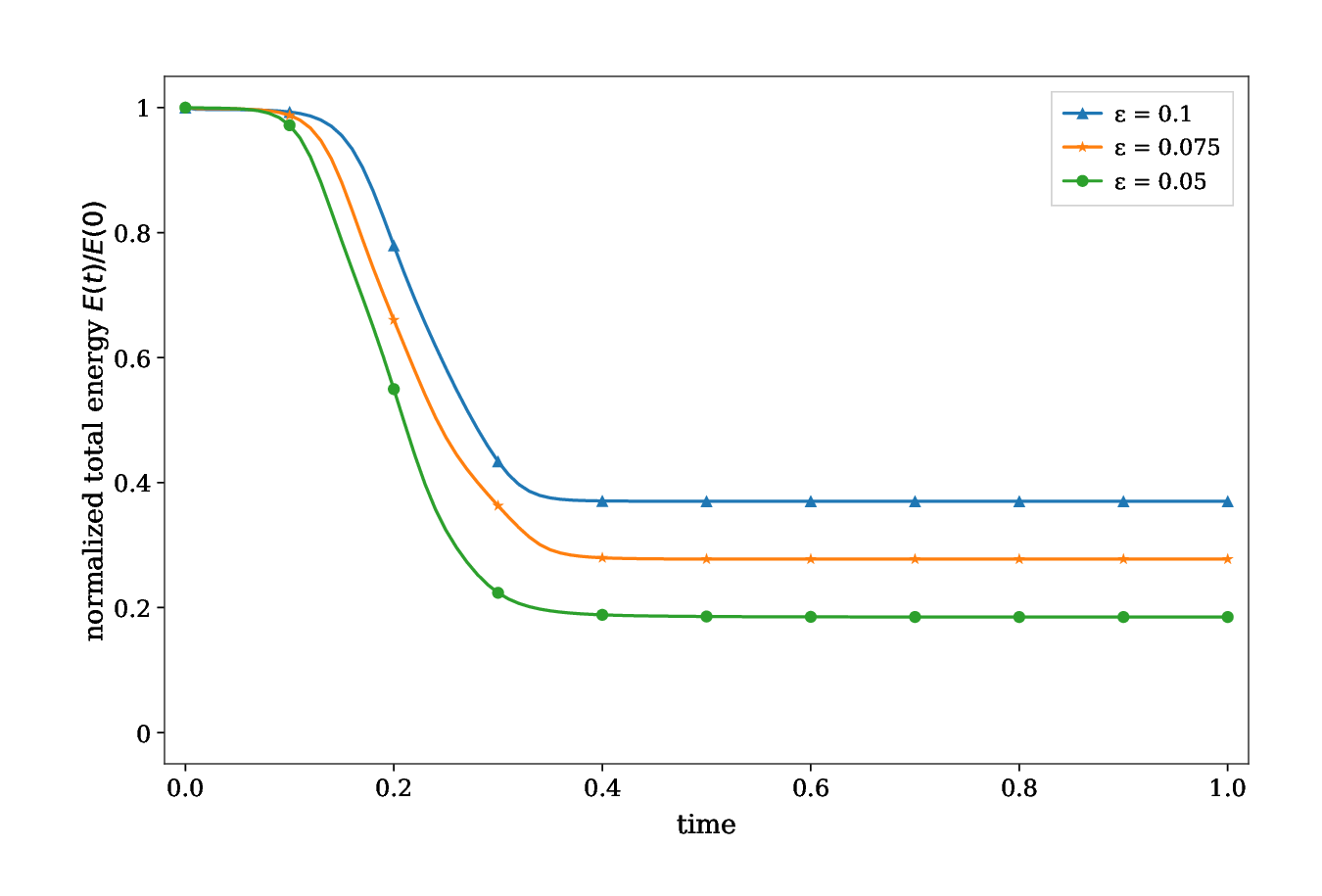}
    }\hspace{0.04\textwidth}%
    \subfigure[Second-order spatial convergence of the total energy (log-log plot).\label{fig:combined-ch-energy-convergence-3d}]{
      \includegraphics[hiresbb,trim={0.126425bp 9.481714bp 3.271968bp -7.811143bp},clip,width=0.35\textwidth,height=0.245\textwidth]{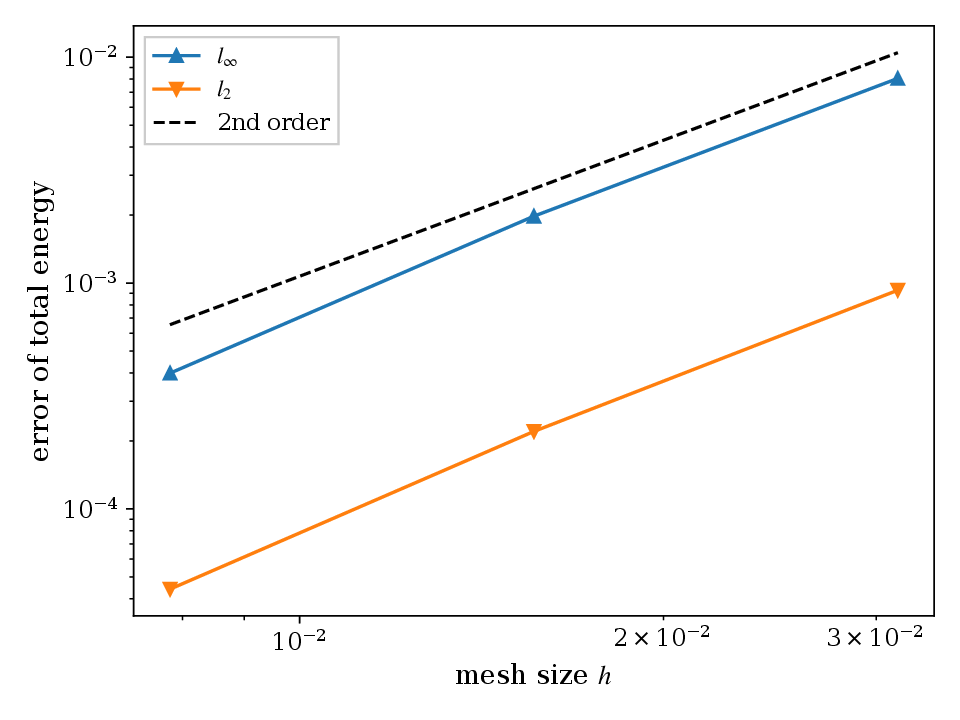}
    }%
  }

  \par\bigskip

  \makebox[\textwidth][c]{%
    \subfigure[Relative mass errors after projection.\label{fig:combined-ch-mass-conservation-3d}]{
      \includegraphics[hiresbb,trim={4.320000bp 18.195429bp 26.537143bp 15.685714bp},clip,width=0.35\textwidth,height=0.245\textwidth]{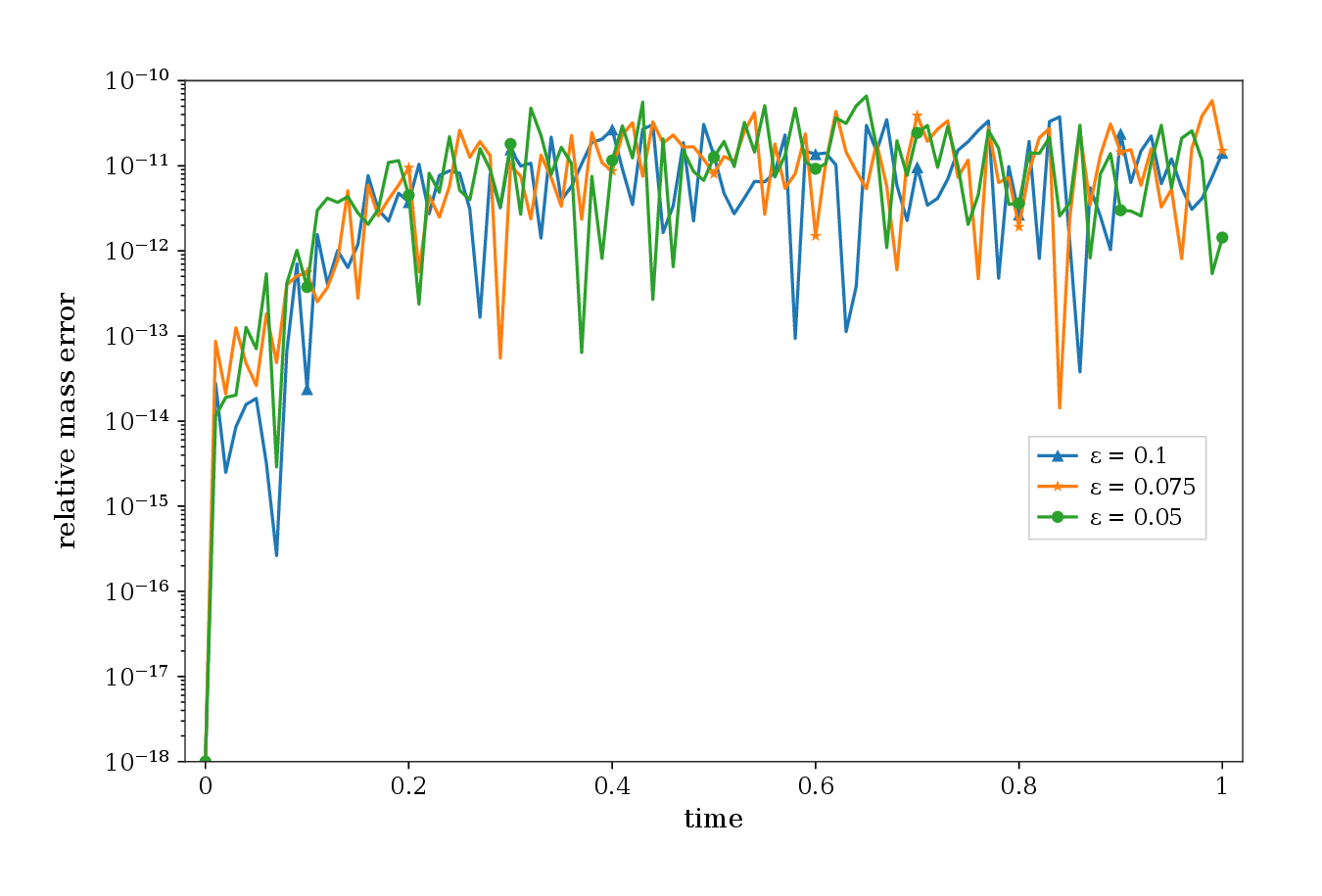}
    }\hspace{0.04\textwidth}%
    \subfigure[Relative energy perturbation induced by the mass projection.\label{fig:combined-ch-projection-energy-3d}]{
      \includegraphics[hiresbb,trim={4.320000bp 18.195429bp 26.537143bp 15.685714bp},clip,width=0.35\textwidth,height=0.245\textwidth]{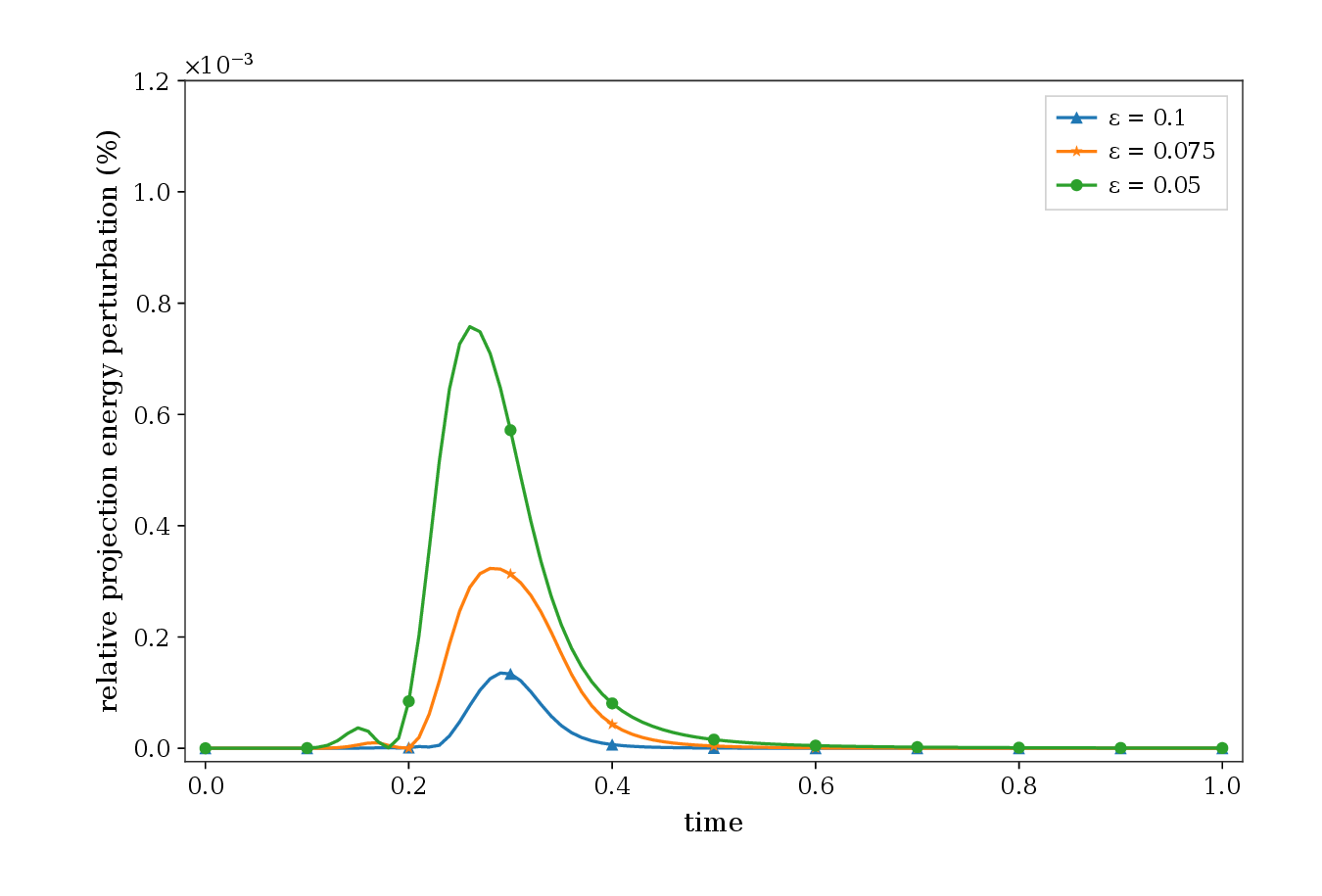}
    }%
  }
  \caption{Energy behavior and mass conservation for the source-free three-dimensional Cahn--Hilliard equation in Example~8.}
  \label{fig:combined-ch-energy-mass-3d}
\end{figure}

\section{Conclusions and Discussion}

This work develops a unified Cartesian-grid KFBI framework for the Allen--Cahn and Cahn--Hilliard equations with homogeneous no-flux boundary conditions on two- and three-dimensional irregular domains.  Each Allen--Cahn update reduces to a Neumann modified Helmholtz problem. For the Cahn--Hilliard equation, an auxiliary-variable reformulation reduces each update to shifted Neumann modified Helmholtz subproblems while avoiding explicit evaluation of the Laplacian of the nonlinear source. These subproblems are handled by the same KFBI solver. In the complex-conjugate case, the real-valued update is reconstructed from a single complex KFBI solve. 

For the Cahn--Hilliard equation, a geometry-weighted projection restores the prescribed discrete mass after each KFBI update. The quadrature weights account for the irregular geometry, and the mass calculation uses only values at physical grid nodes. The projection minimizes the correction in the associated discrete $L^2$ norm subject to the mass constraint, yielding a spatially
uniform shift.

Numerical experiments on two- and three-dimensional irregular domains exhibit second-order spatial accuracy, while paired temporal-error indicators exhibit first-order convergence, consistent with the first-order IMEX discretization. Source-free simulations capture phase evolution and show dissipation of the discrete energy. For the Cahn--Hilliard equation, the projection maintains the prescribed discrete mass, and the energy perturbation introduced by the correction remains small relative to the observed dissipation.

The constant mass correction leaves the discrete gradient energy unchanged, but its effect on the bulk energy can have either sign. The observed energy dissipation therefore provides numerical evidence of dissipative behavior without establishing unconditional energy stability of the fully discrete projected scheme. Future work will focus on incorporating mass conservation directly into the KFBI discretization and developing a fully discrete analysis that couples mass preservation with energy dissipation.

\section*{Acknowledgments}
W. Ying is supported by the National Natural Science Foundation of China in the Division of Mathematical Sciences (Project No. 12471342), partially supported by the National Key R\&D Program of China (Project No. 2024YFA1012403) and the fundamental research funds for the central universities.

\bibliography{math}
\bibliographystyle{siam}

\end{document}